\documentclass[11pt]{article}

\usepackage[a4paper,left=3cm,right=3cm,top=3cm,bottom=3cm,bindingoffset=5mm]{geometry}

\usepackage[utf8]{inputenc}

\usepackage{amsmath, amsthm, amssymb}
\usepackage{color,hyperref}
\usepackage[all,cmtip]{xy}

\usepackage{enumerate}

\usepackage{color}

\newtheorem{mydef}{Definition}[section]
\newtheorem{thm}[mydef]{Theorem}
\newtheorem{prop}[mydef]{Proposition}
\newtheorem{lem}[mydef]{Lemma}
\newtheorem{conj}[mydef]{Conjecture}
\newtheorem{cor}[mydef]{Corollary}

\theoremstyle{remark}
\newtheorem{rem}[mydef]{Remark}
\newtheorem{ex}[mydef]{Example}
\newtheorem*{notations}{Notations}
\newtheorem*{ac}{Acknowledgments}

\title{Absolutely special subvarieties and absolute Hodge cycles}
\author{Tobias Kreutz}
\date{\vspace{-5ex}}

\begin{document}
\maketitle
\begin{abstract}
We introduce the notion of dR-absolutely special subvarieties in motivic variations of Hodge structure as special subvarieties cut out by (de Rham-)absolute Hodge cycles and conjecture that all special subvarieties are dR-absolutely special.
This is implied by Deligne's conjecture that all Hodge cycles are absolute Hodge cycles, but is a much weaker conjecture.
We prove our conjecture for subvarieties satisfying a simple monodromy condition introduced in \cite{KOU}.
We study applications to typical respectively atypical intersections and $\bar{\mathbb{Q}}$-bialgebraic subvarieties. Finally, we show that Deligne's conjecture as well as ours can be reduced to the case of special points in motivic variations.
\end{abstract}
\section*{Introduction}
Let $f:X \to S$ be a smooth projective morphism of smooth (irreducible) quasi-projective algebraic varieties over $\mathbb{C}$.
For any integer $k \ge 0$ the $k$-th primitive cohomology of this family gives rise to a polarizable $\mathbb{Q}$-variation of Hodge structure
$(\mathbb{V}, \mathcal{V}, \nabla, F^{\bullet})$, where
$\mathbb{V} = R_{prim}^kf _{*}^{an} \underline{\mathbb{Q}}$ is a $\mathbb{Q}$-local system on $S^{an}$,
$(\mathcal{V} =R_{prim}^kf_{*} \Omega_{X/S}^{\bullet}, \nabla)$ is the associated algebraic vector bundle with flat connection and $F^{\bullet}$ is the Hodge filtration of the family.
We will usually drop the additional data from the notation and simply denote the variation of Hodge structure by $\mathbb{V}$.
Let $Z \subset S$ be a closed irreducible subvariety.
After choosing a point $s \in Z(\mathbb{C})$, one defines the generic Mumford-Tate group $G_Z$ of $Z$ as the subgroup of $GL(\mathbb{V}_s)$ defined by the condition that it fixes all generic Hodge tensors on $Z$ (a different choice of $s$ defines a conjugated subgroup).

The variation $\mathbb{V}$ gives rise to interesting algebraic subvarieties of $S$, the so-called \emph{special subvarieties} of $S$.

\begin{mydef}[\cite{KO}, Definition 1.2] \label{defspintro1}
A closed irreducible subvariety $Z \subset S$ is called \emph{special} for $\mathbb{V}$ if it is maximal among the closed irreducible algebraic subvarieties of $S$ having the same generic Mumford-Tate group as $Z$.
\end{mydef}
Thus special subvarieties are varieties which are cut out by Hodge cycles.
Although the definition of Hodge cycles is purely analytic,
the Hodge conjecture asserts that all Hodge cycles come from algebraic cycles. We therefore expect that Hodge cycles, seen as classes in the algebraic vector bundle $\mathcal{V}$, are preserved under algebraic field automorphisms of $\mathbb{C}$. This motivates the definition of absolute Hodge cycles by Deligne.
For any automorphism $\sigma \in \mathrm{Aut}(\mathbb{C} / \mathbb{Q})$ we can form the conjugate $S^{\sigma}:= S \otimes_{\mathbb{C}, \sigma} \mathbb{C}$. The projection defines a map $\sigma^{-1}: S \otimes_{\mathbb{C}, \sigma} \mathbb{C} \to S$. Similarly as above, from the conjugated family $f^{\sigma}: X^{\sigma} \to S^{\sigma}$ we obtain a variation of Hodge structure $(\mathbb{V}^{\sigma}, \mathcal{V}^{\sigma}, \nabla^{\sigma}, F^{\bullet, \sigma})$.
The fact that de Rham cohomology can be defined algebraically provides comparison isomorphisms $$ \iota_{\sigma}: (\mathcal{V}^{\sigma} , \nabla^{\sigma}, F^{\bullet, \sigma}) \cong {\sigma^{-1}}^{*} (\mathcal{V}, \nabla, F^{\bullet})$$ of the algebraic filtered vector bundles with connection over $S^{\sigma}$.
We call a collection of variations of Hodge structure $(\mathbb{V}^{\sigma})_{\sigma}$ satisfying the above comparison an \emph{absolute variation of Hodge structure}. An absolute variation of Hodge structure over $S$ is called \emph{of geometric origin} if it arises from a smooth projective family $f:X \to S$ as above.
When given an absolute variation of Hodge structure, following Deligne, one can define a notion of dR-absolute Hodge tensor as a Hodge tensor $\alpha \in \mathbb{V}_s^{\otimes}$ such that, for any $\sigma \in \mathrm{Aut}(\mathbb{C} / \mathbb{Q}) $, the $\sigma$-conjugate of the de Rham component of $\alpha$ comes from a Hodge tensor in $(\mathbb{V}^{\sigma}_{s^{\sigma}})^{\otimes} $ via the comparison isomorphism $\iota_{\sigma}$.

\begin{rem}
The notion of absolute Hodge cycle introduced in Deligne's work \cite{Deligne} is stronger than the version we use here, in the sense that his notion includes a condition on the $\ell$-adic components of Hodge cycles as well. Here we only consider de Rham absolute Hodge cycles in the terminology of (\cite{Voisin}, Definition 0.1). In the following, we will write dR-absolute Hodge for this notion.
\end{rem}

For a subvariety $Z\subset S$ we define the generic dR-absolute Mumford-Tate group $G_Z^{AH}$ to be the subgroup of $GL(\mathbb{V}_s)$ defined by the condition that it fixes all generic dR-absolute Hodge tensors over $Z$.

Since the Hodge conjecture asserts that all Hodge cycles come from algebraic cycles, and an automorphism $\sigma \in \mathrm{Aut}(\mathbb{C} / \mathbb{Q})$ maps algebraic cycles to algebraic cycles, the Hodge conjecture implies that all Hodge cycles are \emph{dR-absolute Hodge}.

\begin{conj}[\cite{Deligne}]\label{conjdel1}
All Hodge cycles are dR-absolute Hodge, i.e. $G_Z = G_Z^{AH}$.
\end{conj}

Suppose the family $f:X\to S$ is defined over a number field $K$.
Then the filtered vector bundle with connection $(\mathcal{V}, \nabla, F^{\bullet})$ is also defined over $K$. In this case, we call the absolute variation a \emph{$\bar{\mathbb{Q}}$-absolute variation}.
It is evident from the definition that the de Rham comparison induces a canonical isomorphism $$G_Z^{AH} \otimes_{\mathbb{Q}} \mathbb{C} \cong G_{Z^{\sigma}}^{AH}\otimes_{\mathbb{Q}} \mathbb{C}$$ of group schemes over $\mathbb{Q}$.
Therefore, assuming Conjecture \ref{conjdel1}, if $Z \subset S$ is special then $Z^{\sigma} \subset S^{\sigma}$ is also a special subvariety for every $\sigma$.
Hence Deligne's conjecture \ref{conjdel1} together with the fact that there are only countably many special subvarieties implies that special subvarieties are defined over $\bar{\mathbb{Q}}$ and all of their finitely many $\mathrm{Gal}(\bar{\mathbb{Q}} / K)$-conjugates are again special.

Still, Deligne's conjecture is much stronger: as explained in (\cite{Voisin}, Lemma 1.4), it is equivalent to the statement that the \emph{locus of Hodge classes} in the algebraic vector bundle $\mathcal{V}$ is defined over $\bar{\mathbb{Q}}$ and its Galois conjugates are also contained in the locus of Hodge classes.
We will now formulate a much weaker conjecture which still implies the definability of special subvarieties over $\bar{\mathbb{Q}}$.

The group-theoretic nature of Definition \ref{defspintro1} invites us to make a variant of it, replacing the generic Mumford-Tate group by the generic dR-absolute Mumford-Tate group.

\begin{mydef}[see Definition \ref{absspdef}]\label{introdefabssp}
A closed irreducible algebraic subvariety $Z$ is called \emph{dR-absolutely special} if it is maximal among the closed irreducible algebraic subvarieties of $S$ with generic dR-absolute Mumford-Tate group $G_Z^{AH}$.
\end{mydef}
In particular, we can think of dR-absolutely special subvarieties as subvarieties cut out by \emph{dR-absolute} Hodge cycles
(this is literally true, see the interpretation in terms of period maps in Proposition \ref{absspperiod}).

Note that dR-absolutely special subvarieties satisfy all the good arithmetic properties that are conjectured to hold for all special subvarieties: they are defined over number fields, and their Galois translates are again special subvarieties.

Deligne's conjecture \ref{conjdel1} immediately implies
\begin{conj}\label{conjabssp}
If $Z \subset S$ is a special subvariety, then $Z$ is dR-absolutely special.
\end{conj}

Except for Deligne's proof in the case of (families of) abelian varieties (\cite{Deligne}, Theorem 2.11), Conjecture \ref{conjdel1} seems completely out of reach, and we have very little knowledge about the existence of dR-absolute Hodge cycles in general variations.
Considering that, it is perhaps surprising that we can prove Conjecture \ref{conjabssp} in a number of cases.
We do so for subvarieties which satisfy a simple monodromy condition introduced in \cite{KOU}.
For a closed irreducible subvariety $Z \subset S$ we define the algebraic monodromy group $H_Z$ as the connected component of the identity of the Zariski closure of the image of the monodromy representation $\rho_Z: \pi_1(Z^{an},s) \to GL(\mathbb{V}_s) $ corresponding to the restriction of the local system $\mathbb{V}$ to $Z^{an}$.

\begin{mydef}[\cite{KOU}, Definition 1.10]
A closed irreducible subvariety $Z \subset S$ is called \emph{weakly non-factor} if it is not contained in a closed irreducible $Y \subset S$ such that $H_{Z}$ is a strict normal subgroup of $H_{Y}$.
\end{mydef}

\begin{thm}[see Corollary \ref{swnf}]\label{introspwnfabssp}
Assume $Z \subset S$ is a special subvariety which is weakly non-factor. Then $Z$ is dR-absolutely special.
\end{thm}

This justifies the result (\cite{KOU}, Theorem 1.12) that weakly non-factor special subvarieties are defined over $\bar{\mathbb{Q}}$ and their Galois conjugates are special.
Examples of weakly non-factor special subvarieties were given in (\cite{KOU}, Corollary 1.13).
An irreducible subvariety $Z$ is called of positive period dimension if $H_Z \not= 1$, or equivalently, the image of $Z$ under the period map is not a point.

\begin{cor}[see Corollary \ref{exampleKOU}]\label{intromaxspabssp}
Suppose that the adjoint group $G_{S}^{ad}$ is simple.
Let $Z \subset S$ be a strict special subvariety which is of positive period dimension and maximal for these properties.
Then $Z$ is dR-absolutely special.
\end{cor}

One might wonder about the difference between Conjecture \ref{conjabssp} and Deligne's conjecture that Hodge classes are dR-absolute Hodge.
We want to emphasize the idea that Conjecture \ref{conjabssp} is a more \emph{geometric} statement than Conjecture \ref{conjdel1}.
While Deligne's conjecture predicts that there is no group $G$ with $G_Z \subsetneq G \subset G_Z^{AH}$, Conjecture \ref{conjabssp} only implies a corresponding geometric statement, namely that no such group $G$ can arise as the generic Mumford-Tate group of a closed irreducible algebraic subvariety $Y$ containing $Z$.

\begin{prop}[see Proposition \ref{nogroup}]\label{intronogroup}
Let $Z$ be dR-absolutely special.
Then there does not exist a closed irreducible algebraic subvariety $Y \supset Z$ with the property
$$G_Z \subsetneq G_Y \subset G^{AH}_Z.$$
\end{prop}

The above philosophy suggests that Conjecture \ref{conjabssp} is particularly strong in situations where the geometry of $Z$ is very closely related to the group-theoretic properties of the period domain.
These cases are known in Hodge theory as \emph{typical intersections} (cf. \cite{BKU}, Definition 1.4).
In Definition \ref{defect} we introduce an integer $\delta_{AH}(Z)$, called the absolute Hodge defect of $Z$, which measures the difference of the dimensions of the Mumford-Tate domains attached to $G_Z$ and $G_Z^{AH}$. Conjecture \ref{conjdel1} predicts that $\delta_{AH}(Z)=0$ for all subvarieties $Z$.
We show in Theorem \ref{bound} that the absolute Hodge defect can be bounded in terms of the Hodge codimension of $S$ and $Z$ as introduced in (\cite{BKU}, Definition 4.1).
This shows that the closer $Z$ is to being a typical intersection, the stronger Conjecture \ref{conjabssp} becomes in relation to Conjecture \ref{conjdel1}.

\begin{thm}[see Theorem \ref{bound}]\label{introbound}
Assume that $G_S =G_S^{AH}$, i.e. all generic Hodge cycles on $S$ are dR-absolute Hodge.
Let $Z$ be a dR-absolutely special subvariety. Then
$$\delta_{AH}(Z) \le Hcd(S)-Hcd(Z).$$
In particular, if $Z$ is a typical special subvariety then $\delta_{AH}(Z)=0$.
\end{thm}

Let us give a few applications of dR-absolutely special subvarieties.
First we use a recent result of Baldi-Klingler-Ullmo (cf. \cite{BKU}, Proposition 2.4) to give new examples of dR-absolutely special subvarieties. The \emph{level} of a variation of Hodge structure $\mathbb{V}$ as defined in (\cite{BKU}, Definition 3.13) is a measure for the length of the Hodge filtration on the Lie algebra of the adjoint generic Mumford-Tate group.

\begin{thm}[see Theorem \ref{maxatypabssp}]\label{intromaxatypabssp}
Suppose $G_S^{ad}$ is simple and the level of $\mathbb{V}$ is greater or equal to two. Then any maximal atypical special subvariety $Z \subset S$ of positive period dimension is dR-absolutely special.
\end{thm}

As a generalization of (\cite{UY}, Theorem 1.4) one conjectures (cf. \cite{KlinglerICM}, 4.1) that $\bar{\mathbb{Q}}$-bialgebraic subvarieties in a $\bar{\mathbb{Q}}$-absolute variation of Hodge structure of geometric origin are special (see Definition \ref{defbialg} for a definition of $\bar{\mathbb{Q}}$-bialgebraic subvarieties).
We prove that this is indeed the case for maximal $\bar{\mathbb{Q}}$-bialgebraic subvarieties of positive period dimension in case the adjoint group of $G_S$ is simple.

\begin{thm}[see Theorem \ref{bialgmax}]\label{introbialgmax}
Suppose $G_S^{ad}$ is simple.
Let $Z \subset S$ be a strict maximal $\bar{\mathbb{Q}}$-bialgebraic subvariety of positive period dimension. Then $Z$ is a dR-absolutely special subvariety (and in particular a special subvariety).
\end{thm}

We define a class of motivic variations of Hodge structure as those absolute variations which look as if they come from geometry.

\begin{mydef}
An absolute variation of Hodge structure $(\mathbb{V}^{\sigma})_{\sigma}$
on $S$ is called a \emph{motivic variation of Hodge structure} if there exists a dense open subvariety $U \subset S$ such that for all points $x \in U(\mathbb{C})$, the collection of Hodge structures $(\mathbb{V}_{x^{\sigma}}^{ \sigma})_{\sigma} $ together with their comparison isomorphisms is the realization of a motive for dR-absolute Hodge cycles.
A $\bar{\mathbb{Q}}$-absolute variation of Hodge structure with this property is called $\bar{\mathbb{Q}}$-\emph{motivic} variation of Hodge structure.
\end{mydef}
\begin{rem}
We refer to (\cite{DM}, §6) for the construction of the category of motives for absolute Hodge cycles. Again, we consider a variant of it, using only dR-absolute Hodge cycles. Every absolute variation of geometric origin is an example of such a motivic variation.
\end{rem}

It is reasonable to expect that Deligne's Conjecture \ref{conjdel1}, and therefore also Conjecture \ref{conjabssp} should hold for subvarieties in $\bar{\mathbb{Q}}$-motivic variations.
As a final application, we use monodromy arguments to show that these conjectures for $\bar{\mathbb{Q}}$-absolute variations of geometric origin can be reduced to the case of special points in $\bar{\mathbb{Q}}$-motivic variations.

\begin{thm}[see Theorem \ref{reducepoints}] \label{introreducepoints}
\textnormal{ }
\begin{enumerate}[(i)]
\item \label{introreducedeligne}
Suppose Deligne's conjecture \ref{conjdel1} holds for special points in $\bar{\mathbb{Q}}$-motivic variations. Then it holds for all subvarieties in $\bar{\mathbb{Q}}$-absolute variations of geometric origin.
\item \label{introreduceabssp}
Suppose that Conjecture \ref{conjabssp} holds for special points in $\bar{\mathbb{Q}}$-motivic variations. Then it holds for all special subvarieties in $\bar{\mathbb{Q}}$-absolute variations of geometric origin.
\item \label{introfodgalconj}
Suppose that special points in $\bar{\mathbb{Q}}$-motivic variations are defined over $\bar{\mathbb{Q}}$ and all their Galois conjugates are special. Then the same holds for all special subvarieties in $\bar{\mathbb{Q}}$-absolute variations of geometric origin.
\item \label{introreducebialg}
Suppose that all $\bar{\mathbb{Q}}$-bialgebraic points in $\bar{\mathbb{Q}}$-motivic variations are dR-absolutely special. Then all $\bar{\mathbb{Q}}$-bialgebraic subvarieties in $\bar{\mathbb{Q}}$-absolute variations of geometric origin are dR-absolutely special.
\end{enumerate}
\end{thm}

\begin{rem}
The part of Theorem \ref{introreducepoints}(\ref{introfodgalconj}) concerning definability over $\bar{\mathbb{Q}}$ was already proven in (\cite{KOU}, Corollary 1.14), but the part on Galois conjugates was left open, cf. (\cite{KOU}, Remark 3.5).
\end{rem}

\begin{notations}

All varieties in this paper are assumed to be reduced.
Unless stated otherwise, by a subvariety $Z$ of a complex algebraic variety $S$ we mean a closed irreducible complex algebraic subvariety.
If $S$ is defined over a number field $K \subset \mathbb{C}$ and $Z$ is defined over an extension $L $ of $K$, we denote by $Z_L$ the associated $L$-variety.
If a distinction is necessary, we will refer to $Z$ considered as a complex variety by $Z_{\mathbb{C}} = Z_L \otimes_L \mathbb{C}$.
Throughout this paper, by a Hodge cycle we mean a $(0,0)$-cycle.
\end{notations}
\begin{ac}
The author would like to thank Bruno Klingler for many helpful discussions and comments.
\end{ac}

\section{Variations of Hodge structure and special subvarieties}

In this section we recall fundamental facts from Hodge theory that will be used throughout the paper.
Let $S$ be a smooth irreducible quasi-projective algebraic variety over $\mathbb{C}$, and $\mathbb{V}_{\mathbb{Z}}$ a (pure, polarizable) $\mathbb{Z}$-variation of Hodge structure on $S$. Denote by $\mathbb{V}:= \mathbb{V}_{\mathbb{Z}} \otimes_{\mathbb{Z}} \mathbb{Q}$ the associated
$\mathbb{Q}$-variation of Hodge structure.

\subsection{Hodge varieties and special subvarieties}

At any point $s \in S(\mathbb{C})$ the fiber $\mathbb{V}_s$ carries a polarizable Hodge structure.
For every $m,n \ge 0$ we define the tensor Hodge structure
$\mathbb{V}_{s}^{\otimes (m,n)} := \mathbb{V}_s^{\otimes m} \otimes \mathbb{V}_s^{\vee \otimes n}$ and $\mathbb{V}_s^{\otimes} := \bigoplus_{m,n \ge 0} \mathbb{V}_s^{\otimes (m,n)}$.

To a closed irreducible algebraic subvariety $Z \subset S$ one can attach an important Hodge-theoretic invariant: its generic Mumford-Tate group with respect to the variation $\mathbb{V}$.
After choosing a point $s \in Z(\mathbb{C})$, let $\mathcal{H}_Z \subset \mathbb{V}_s^{\otimes} $ be the subspace of those $v \in \mathbb{V}_s^{\otimes}$ satisfying the condition that every parallel transport of $v$ is a Hodge cycle at every point of $Z$.

\begin{mydef}
The generic Mumford-Tate group $G_Z$ of $Z$ is defined to be
the subgroup of $GL(\mathbb{V}_s)$ fixing the tensors in $\mathcal{H}_Z$.
\end{mydef}

\begin{rem}\textnormal{ }
\begin{enumerate}[(i)]
\item
A priori, our definition of $G_Z$ depends on the choice of $s \in Z(\mathbb{C})$. However, as we vary the point $s$, the associated groups form a local system of algebraic groups over $Z$, which shows that the group $G_Z$ is independent of the choice of $s$ up to monodromy.
In fact, for $s\in Z(\mathbb{C})$ away from a countable union of closed algebraic subvarieties we can identify $G_Z$ with the Mumford-Tate group of the Hodge structure $\mathbb{V}_s$.
From this we see that $G_Z$ is a connected reductive group over $\mathbb{Q}$.
\item
The generic Mumford-Tate group $G_Z$ is often defined by restricting to the smooth locus of $Z$. We show in Remark \ref{smoothpart} that our definition gives the same group.
\end{enumerate}\end{rem}

An important feature of variations of Hodge structure is the fact that they have a geometric interpretation in terms of period maps:
to the variation of Hodge structure $\mathbb{V}$ one can naturally attach a holomorphic period map $$\Phi: S^{an} \to \Gamma \backslash \mathcal{D}$$ with target a so-called Hodge variety.
We now recall the most important aspects of this notion.

\begin{mydef}[\cite{Klingleratypical}, 3.1]
\textnormal{ }
\begin{enumerate}[(i)]
\item A \emph{(connected) Hodge datum} is a pair $(G, \mathcal{D})$ consisting of a reductive group $G$ over $\mathbb{Q}$ and a connected component $\mathcal{D}$ of the $G(\mathbb{R})$-conjugacy class of a Hodge cocharacter $h: \mathbb{S} \to G_{\mathbb{R}}$ satisfying the following conditions:
\begin{enumerate}
\item
the weight homomorphism $w: \mathbb{G}_{m, \mathbb{R}} \subset \mathbb{S} \overset{h}{\to} G_{\mathbb{R}}$ is defined over $\mathbb{Q}$ and factors through the center of $G$,
\item
the involution induced by $h(i)$ is a Cartan involution of $G_{\mathbb{R}}^{ad}$.
\end{enumerate}
\item
A \emph{morphism of Hodge data} $(G, \mathcal{D}) \to (G', \mathcal{D}')$ is a morphism of reductive groups $G \to G'$ such that $\mathcal{D}$ maps to $\mathcal{D}'$.

\end{enumerate}
\end{mydef}

\begin{mydef}[\cite{Klingleratypical}, 3.3]
\textnormal{ }
\begin{enumerate}[(i)]
\item A \emph{(connected) Hodge variety} attached to the Hodge datum $(G, \mathcal{D})$ is a quotient $\Gamma\backslash \mathcal{D} $ of $\mathcal{D}$ by an arithmetic subgroup $\Gamma \subset G(\mathbb{Q})$. It is naturally a complex analytic space, and becomes a complex manifold after possibly replacing $\Gamma$ by a finite index subgroup.
\item
A \emph{morphism of Hodge varieties} $\Gamma\backslash \mathcal{D} \to \Gamma' \backslash \mathcal{D}'$ is a morphism induced by a morphism of Hodge data $(G, \mathcal{D}) \to (G', \mathcal{D}')$ under which $\Gamma$ gets mapped to $\Gamma'$.
\item
A \emph{special subvariety} of a Hodge variety $\Gamma\backslash \mathcal{D}$ is the image of a morphism of Hodge varieties.
\end{enumerate}
\end{mydef}

From a variation of Hodge structure $\mathbb{V}$, one can construct (after possibly replacing $S$ by a finite étale cover) a period map

$$\Phi: S^{an} \to X:= \Gamma\backslash \mathcal{D},$$
where $(G_S, \mathcal{D})$ is the generic Hodge datum on $S$.

Let $Z \subset S$ be a closed irreducible subvariety. The inclusion $G_Z \subset G_S$ gives rise to a morphism of Hodge data $(G_Z, \mathcal{D}_{Z}) \to (G_S, \mathcal{D})$, and the image $X_{Z}$ is a special subvariety of $X$.
Then $\Phi(Z^{an}) \subset X_{Z}$ and $X_{Z}$ is the smallest special subvariety of $X$ containing $\Phi(Z^{an})$.
\begin{rem}\label{smoothpart}
We can now justify our definition of the generic Mumford-Tate group in the following way:
Let $G_{Z^{sm}}$ be the generic Mumford-Tate group in restriction to the smooth locus $Z^{sm}$ of $Z$ and $X_{Z^{sm}}$ the associated special subvariety of $X$.
Then $$\Phi(Z^{sm,an}) \subset X_{Z^{sm}}$$ and since the smooth locus is dense in $Z^{an}$ (even for the analytic topology) and $X_{{Z^{sm}}}$ is a closed analytic subspace of $X$, we get
$$\Phi(Z^{an}) \subset X_{Z^{sm}}.$$
From this we see that every generic Hodge tensor on $Z^{sm}$ extends to $Z$, i.e. $G_Z = G_{Z^{sm}}$.
\end{rem}

The subvarieties of $S$ which are maximal with a given generic Mumford-Tate group are of particular Hodge-theoretic importance.

\begin{mydef}[\cite{KO}, Definition 1.2] \label{definitionspecial1}
A closed irreducible algebraic subvariety $Z \subset S$ is called \emph{special} for $\mathbb{V}$ if it is maximal among the closed irreducible algebraic subvarieties of $S$ having the same generic Mumford-Tate group as $Z$.
\end{mydef}

A famous result of Cattani-Deligne-Kaplan (and recently reproved using o-minimal methods by Bakker-Klingler-Tsimerman in \cite{BKT}) describes the special subvarieties of $S$ using the period map.

\begin{thm}[\cite{CDK}, Corollary 1.3]
The special subvarieties of $S$ are precisely the irreducible components of the preimages of special subvarieties in $X$ under the period map.
\end{thm}

\subsection{Monodromy}

When we forget about the Hodge filtration and just consider the underlying local system of the variation, we enter the world of monodromy.

\begin{mydef}\label{algmon}
The \emph{algebraic monodromy group} $H_Z$ of $Z$ is defined to be the identity component of the
Zariski closure of the image of the monodromy representation $\rho_Z: \pi_{1}(Z^{an},s) \to GL(\mathbb{V}_s)$.
\end{mydef}

Note that our subvariety $Z \subset S$ will often be singular, and that usually the fundamental group of a singular analytic space is not particularly well-behaved. In our case however, we will use that $Z$ is a closed subvariety of the smooth variety $S$ and the local system $\mathbb{V}$ moreover supports a variation of Hodge structure. In this situation we see that in the definition of the algebraic monodromy group $H_Z$ we may always replace $Z$ by its smooth locus.

\begin{lem}\label{monsm}
Let $Z \subset S$ be a closed irreducible subvariety and let $Z^{sm}$ be the smooth locus of $Z$. Then $H_Z = H_{Z^{sm}}$, i.e. every global section of some $\mathbb{V}^{\otimes (m,n)}$ over a finite étale cover of $Z^{sm}$ extends uniquely to a global section over a finite étale cover of $Z$.
\end{lem}
\begin{proof}
Let $$\iota: X_{G_Z}:= \Gamma_Z \backslash \mathcal{D}_{Z} \to X= \Gamma \backslash \mathcal{D}$$ denote the morphism of Hodge varieties induced from the morphism of Hodge data $(G_Z, \mathcal{D}_Z) \to (G_S, \mathcal{D})$.
By (\cite{Andre}, Theorem 1) the group $H_{Z^{sm}}$ is a normal subgroup of $G_{Z^{sm}}= G_Z$. The projection to the quotient gives rise to a morphism of Hodge data $$(G_{Z}, \mathcal{D}_{Z}) \to (G_{Z}/H_{Z^{sm}}, \overline{\mathcal{D}}_{Z}),$$
and we denote the associated morphism of Hodge varieties by
$$\pi: X_{G_Z} = \Gamma_Z \backslash \mathcal{D}_{Z} \to X_{G_Z/H_{Z^{sm}}} := \overline{\Gamma}_Z \backslash \overline{\mathcal{D}}_Z.$$

After possibly replacing $Z^{sm}$ by a finite étale cover, we get a period map $\Phi': Z^{sm,an} \to X_{G_Z}$.
By (\cite{KO}, Lemma 4.12) the projection of $\Phi'(Z^{sm,an})$ to the quotient Hodge variety $$X_{G_Z/H_{Z^{sm}}} = \overline{\Gamma}_Z \backslash \overline{\mathcal{D}}_Z$$ is a single point $y$.
Denoting by $$X_{H_{Z^{sm}},y} = \pi^{-1}(\{y\})$$ the preimage of $y$ in $X_{G_Z} $, the smooth locus of
$Z$ is thus contained in the preimage of $\iota(X_{H_{Z^{sm}},y})$ under $\Phi$.
Now since $Z^{sm}$ is dense in $Z$, we deduce that also $Z \subset \Phi^{-1}(\iota(X_{H_{Z^{sm}},y}))$.

Since $Z$ is irreducible, under the lifted period map
\begin{equation}\label{liftedperiodmap} \tilde{\Phi}: \tilde{S} \to \mathcal{D} \end{equation}
for the universal cover $\tilde{S}$ of $S$, the universal cover $\tilde{Z}$ is mapped to $\mathcal{D}_Z$ and its projection to $\overline{\mathcal{D}}_Z$ is a single point $\tilde{y} \in \overline{\mathcal{D}}_Z$.
The lifted period map (\ref{liftedperiodmap}) is $\pi_1(S^{an},s)$-equivariant. This shows that up to a finite subgroup, the image of the monodromy representation $\rho_Z: \pi_1(Z^{an},s) \to GL(\mathbb{V}_s)$
is contained in $G_Z(\mathbb{Q})$, and the projection
\begin{equation} \label{monodromyprojection}\pi_1(Z^{an},s) \to G_Z(\mathbb{Q}) \twoheadrightarrow G_Z(\mathbb{Q}) / H_{Z^{sm}}(\mathbb{Q})\end{equation}
is contained in the stabilizer of the point $\tilde{y}$.
Since this stabilizer is a compact real Lie group, we conclude that the image of (\ref{monodromyprojection}) is finite. Therefore a finite index subgroup of $\rho_Z(\pi_1(Z^{an},s))$ is contained in $H_{Z^{sm}}$, which shows that $H_Z = H_{Z^{sm}}$.
\end{proof}

The following Theorem is a well-known result due to André (in the smooth case, but we can easily reduce to that case using Lemma \ref{monsm} and Remark \ref{smoothpart}):
\begin{prop}[\cite{Andre}, Theorem 1]\label{normal1}
The algebraic monodromy group $H_Z$ is a normal subgroup of the derived group of the generic Mumford-Tate group $G_Z$.
\end{prop}

Just like the generic Mumford-Tate group, the algebraic monodromy group of a variation of Hodge structure has an interpretation in terms of period maps.
The description in the proof of Lemma \ref{monsm} motivates the following definition.

\begin{mydef}
Let $ X \overset{\iota}{\leftarrow} X_1 \overset{\pi}{\rightarrow} X_2$ be a diagram of morphisms of Hodge varieties and $x_2 \in X_2$. An irreducible component of a variety of the form $\iota(\pi^{-1}(\{x_2\}))$ is
called a \emph{weakly special subvariety} of $X$.
\end{mydef}

\begin{mydef}
A closed irreducible algebraic subvariety $Z \subset S$ is called \emph{weakly special} for $\mathbb{V}$ if it is maximal among the closed irreducible algebraic subvarieties of $S$ having the same algebraic monodromy group as $Z$.
\end{mydef}

Weakly special subvarieties admit an interpretation via the period map.
\begin{thm}[\cite{KO}, Corollary 4.14]\label{wsperiodmap1}
The weakly special subvarieties of $S$ are precisely the irreducible components of the preimages of weakly special subvarieties in $X$ under the period map.
\end{thm}

We will frequently use the Theorem of the fixed part. The following is a slight generalization of (\cite{Schmid}, Corollary 7.23) to the non-smooth case.

\begin{thm}[\cite{Schmid}, Corollary 7.23] \label{fixedpartnonsm1}
Let $\xi$ be a global section of $\mathbb{V}^{\otimes (m,n)}$ over a finite étale cover $Z'$ of a closed irreducible subvariety $Z\subset S$ such that $\xi_s$ is a Hodge class for some point $s \in Z'(\mathbb{C})$. Then $\xi_t$ is Hodge for every point $t \in Z'(\mathbb{C})$.
\end{thm}
\begin{proof}
By the Theorem of the fixed part applied to a suitable finite étale cover of the smooth locus of $Z$, we see as in (\cite{KO}, Lemma 4.12) that the period map restricted to $Z^{sm}$ factors as
$$ \Phi: Z^{sm,an} \to \iota(X_{H_Z, y}) \subset \iota(X_{G_Z})$$
where $X_{H_Z,y }:= \pi^{-1}(\{y\})$ and $\iota: X_{G_Z} \to X$ and $\pi: X_{G_Z} \to X_{G_Z/H_Z}$ are morphisms of Hodge varieties.
As above, since $Z^{sm}$ is dense in $Z$, the image of $Z$ under the period map is also contained in $\iota(X_{H_Z, y})$.
In particular, up to replacing $Z$ by some finite étale cover we can form the period map
\begin{equation}\label{periodmap1} \Phi': {Z}^{an} \to X_{G_Z} \to X_{G_Z/H_Z} \end{equation}
and the above inclusion shows that it is constant.
Consider the natural representation of $G_Z/H_Z$ on $(\mathbb{V}_s^{\otimes (m,n)})^{H_Z}$. Under this representation, the period map (\ref{periodmap1}) corresponds to the subvariation of Hodge structure $\left.(\mathbb{V}^{\otimes (m,n)} \right|_{Z})^{H_Z}$ over $Z$, which is therefore constant.
Hence every global section that is a Hodge class at one point is a Hodge class everywhere.
\end{proof}

The following Lemma is a geometric application of the Theorem of the fixed part.
\begin{lem}\label{lem1}
If $Z \subset Y$ are two closed irreducible subvarieties satisfying $H_Y \subset G_Z$, then $G_Z = G_Y$.
\end{lem}
\begin{proof}
Let $s \in Z(\mathbb{C})$.
Since $G_Z$ is reductive, it is enough to show that every fixed tensor $v \in \mathbb{V}_s^{\otimes}$ of $G_Z$ is also fixed by $G_{Y}$. The condition $H_{Y} \subset G_{Z}$ shows that after possibly replacing $Y$ by a finite étale cover, there is a global section on $Y$ extending $v$. Since $v$ is Hodge at the point $s \in Z(\mathbb{C})$, the Theorem of the fixed part (Theorem \ref{fixedpartnonsm1}) allows us to conclude that this global section is a generic Hodge tensor over $Y$. So $G_{Y}$ fixes $v$, as desired.
\end{proof}

As can be easily seen from the description using the period map, every special subvariety is weakly special. We show how Lemma \ref{lem1} can be applied to give a more direct argument.

\begin{lem}\label{simpliesws}
Any special subvariety is weakly special.
\end{lem}
\begin{proof}
Suppose $Z$ is special and $Y \supset Z$ is a closed irreducible subvariety of $S$ which satisfies $H_Y=H_Z \subset G_Z$. By Lemma \ref{lem1} we get $G_Z = G_Y$ which implies $Z=Y$ since $Z$ is special. Hence $Z$ is weakly special.
\end{proof}

\subsection{Fields of definition of weakly special subvarieties}

Suppose $S$ is defined over $\bar{\mathbb{Q}}$.
In this section, we prove that
weakly special subvarieties of $S$ are defined over $\bar{\mathbb{Q}}$ once they contain a $\bar{\mathbb{Q}}$-point.
This recovers and generalizes a result of Saito-Schnell (\cite{SchnellSaito}) for special subvarieties.
Surprisingly, as in their paper, we merely require the base $S$ to be defined over $\bar{\mathbb{Q}}$, there is no condition on the variation $\mathbb{V}$.

Choose a $\bar{\mathbb{Q}}$-point $s \in S(\bar{\mathbb{Q}})$ and a prime number $\ell$.
The monodromy representation $\pi_1(S_{\mathbb{C}}^{an},s) \to GL(\mathbb{V}_{\mathbb{Z},s} \otimes {\mathbb{Z}_{\ell}})$ corresponding to the (analytic) $\mathbb{Z}_{\ell}$-local system $\mathbb{V}_{\mathbb{Z}}\otimes \mathbb{Z}_{\ell}$ factors through the étale fundamental group, giving rise to a continuous representation
$$ \rho_{\ell}: \pi_1^{\acute{e}t}(S_{\bar{\mathbb{Q}}},s) \cong \pi_1^{\acute{e}t}(S_{\mathbb{C}},s) \to GL(\mathbb{V}_{\mathbb{Z},s} \otimes \mathbb{Z}_{\ell}),$$
hence to an étale $\mathbb{Q}_{\ell}$-local system $\mathbb{V}_{\ell}$ on $S_{\bar{\mathbb{Q}}}$.
Here the isomorphism
$\pi_1^{\acute{e}t}(S_{\bar{\mathbb{Q}}},s) \cong \pi_1^{\acute{e}t}(S_{\mathbb{C}},s)$
comes from the fact that the étale fundamental group of a smooth (quasi-projective) variety is invariant under base change of algebraically closed fields (\cite{Grothendieck} Exposé VIII, Proposition 4.6).

For a closed irreducible complex subvariety $Z \subset S$, choose a point $s \in Z(\mathbb{C})$ and define $H_{\ell,Z}$ to be the identity component of the Zariski closure of the image of $\pi_1^{\acute{e}t}(Z, s)$ in $GL(\mathbb{V}_{\ell, s})$.

\begin{prop}\label{above} The comparison isomorphism $\mathbb{V}_{\ell, s} \cong \mathbb{V}_{s} \otimes \mathbb{Q}_{\ell}$ induces an isomorphism $H_{\ell, Z} \cong H_{Z} \otimes \mathbb{Q}_{\ell}$.
\end{prop}
\begin{proof}
We recall the argument from (\cite{Moonen}, Lemma 4.3.4).
The étale fundamental group $\pi_1^{\acute{e}t}(Z,s)$ is the profinite completion of $\pi_1(Z^{an},s)$ (cf. \cite{Grothendieck} Exposé V, Corollary 5.2). It is enough to show that $\rho_{\ell}(\pi_1^{\acute{e}t}(Z,s)) \subset H_Z\otimes \mathbb{Q}_{\ell}$.
This follows from the fact that the image of $\pi_1(Z^{an},s)$ is dense in $\rho_{\ell}(\pi_1^{\acute{e}t}(Z,s))$ for the $\ell$-adic topology, so a fortiori for the Zariski topology.
\end{proof}

\begin{lem}\label{monconj}
We have a natural isomorphism $H_{Z^{\sigma}}\otimes \mathbb{Q}_{\ell} \cong H_{Z}\otimes \mathbb{Q}_{\ell}$ for any $\sigma \in Aut(\mathbb{C} / \bar{\mathbb{Q}})$.
\end{lem}
\begin{proof}
Following the above Proposition \ref{above}, it is enough to prove $H_{\ell, Z} \cong H_{\ell, Z^{\sigma}}$. The projection $\sigma^{-1}: Z^{\sigma} = Z \otimes_{\mathbb{C}, \sigma} \mathbb{C} \cong Z$ is an isomorphism of abstract schemes (but not of varieties over $\mathbb{C}$) and induces the following diagram
$$\xymatrix{\pi_1^{\acute{e}t}(Z, s) \ar[r] & \pi_1^{\acute{e}t}(S_{\mathbb{C}}, s) \ar[r] & GL(\mathbb{V}_{\ell, s}) \\
\pi_1^{\acute{e}t}(Z^{\sigma}, s^{\sigma}) \ar[r] \ar[u]^{\cong}& \pi_1^{\acute{e}t}(S_{\mathbb{C}}, s^{\sigma}) \ar[r] \ar[u]^{\cong}& GL(\mathbb{V}_{\ell, s^{\sigma}}) \ar[u]^{\cong}\, .
}$$
Now the fact that the local system $\mathbb{V}_{\ell}$ is defined over $\bar{\mathbb{Q}}$ translates into the commutativity of the right hand square of the diagram. It follows easily that the entire diagram commutes.
As a consequence, there is a natural isomorphism between the Zariski closure of the image of $\pi_1^{\acute{e}t}(Z, s)$ in $GL(\mathbb{V}_{\ell, s})$ and the Zariski closure of the image of $ \pi_1^{\acute{e}t}(Z^{\sigma}, s^{\sigma})$ in $GL(\mathbb{V}_{\ell, s^{\sigma}})$.
\end{proof}

\begin{lem} \label{wsconj}
If $Z$ is weakly special for $\mathbb{V}$, then $Z^{\sigma}$ is weakly special for $\mathbb{V}$ for any $\sigma \in Aut(\mathbb{C} / \bar{\mathbb{Q}})$.
\end{lem}
\begin{proof}
Let $Y \supset Z^{\sigma}$ be closed irreducible such that $H_{Y} = H_{Z^{\sigma}}$. From Lemma \ref{monconj} we see that the subvariety $Y^{\sigma^{-1}}$ containing $Z$ satisfies $H_{Y^{\sigma^{-1}}} \otimes \mathbb{Q}_{\ell} = H_{Z} \otimes \mathbb{Q}_{\ell}$. This forces $H_{Y^{\sigma^{-1}}} = H_Z$. Now the fact that $Z$ is weakly special implies that $Z= Y^{\sigma^{-1}}$, hence $Z^{\sigma} = Y$. We conclude that $Z^{\sigma}$ is weakly special.
\end{proof}

In contrast to the case of special subvarieties, we cannot expect weakly special subvarieties to be defined over $\bar{\mathbb{Q}}$ in general. For example, if the period map does not contract positive dimensional subvarieties to a point, every point $s \in S(\mathbb{C})$ is weakly special.
However, we now show that they are defined over $\bar{\mathbb{Q}}$ once they contain a single $\bar{\mathbb{Q}}$-point:

\begin{thm}\label{fod}
Let $Z$ be a weakly special subvariety of $S$ containing a $\bar{\mathbb{Q}}$-point. Then $Z$ is defined over $\bar{\mathbb{Q}}$.
\end{thm}
\begin{proof}
We need to show that the set $\{Z^{\sigma}\}_{\sigma \in Aut(\mathbb{C}/\bar{\mathbb{Q}})}$ of conjugates of $Z$ is countable. For $z$ a $\bar{\mathbb{Q}}$-point contained in $Z$, we have $z \in Z^{\sigma}(\mathbb{C})$ for all $\sigma$.
Note that all $Z^{\sigma}$ are again weakly special by Lemma \ref{wsconj}.
By the characterization of weakly special subvarieties in Proposition \ref{wsperiodmap1}, for every $\sigma$ there is a diagram of morphisms of Hodge varieties $ X \overset{\iota_{\sigma}}{\leftarrow} X^{\sigma}_1 \overset{\pi_{\sigma}}{\rightarrow} X^{\sigma}_2$
and a point $x_{\sigma} \in X^{\sigma}_2$ such that $Z^{\sigma}$ is one of the finitely many irreducible components of the preimage of $\iota_{\sigma}(\pi_{\sigma}^{-1}(\{x_{\sigma}\}))$ under the period map.
In particular, $Z^{\sigma}$ is determined (up to choosing one of finitely many irreducible components) by the diagram of Hodge morphisms and the point $x_{\sigma}$.
There are only countably many choices for the diagram of Hodge morphisms, and we need to show that for a given diagram $(\iota, \pi)$ there are only countably many choices for $\iota(\pi^{-1}(\{x_{\sigma}\})) $.
Since $z \in Z^{\sigma}(\mathbb{C})$ for all $\sigma$, the $\iota(\pi^{-1}(\{x_{\sigma}\}))$ all contain a common point.
By noting that the $\pi^{-1}(\{x_{\sigma}\})$ are pairwise disjoint as $x_{\sigma}$ varies, and the group $\Gamma \subset G_S(\mathbb{Q})$ is countable, one sees that this leaves only countably many choices for $\iota(\pi^{-1}(\{x_{\sigma}\}))$.
\end{proof}
As a corollary we obtain a new proof of the following result, proven in (\cite{SchnellSaito}, Theorem 1):
\begin{cor}

Let $Z$ be a special subvariety of $S$ containing a $\bar{\mathbb{Q}}$-point. Then $Z$ is defined over $\bar{\mathbb{Q}}$.
\end{cor}

\section{Absolute variations of Hodge structure}

Let $S$ be a smooth irreducible quasi-projective complex algebraic subvariety.
For every $\sigma \in \mathrm{Aut}(\mathbb{C} / \mathbb{Q})$ we can form the conjugate $S^{\sigma} = S \otimes_{\mathbb{C}, \sigma} \mathbb{C}$. We use the notation $\sigma^{-1}$ for the projection $\sigma^{-1}: S^{\sigma} = S \otimes_{\mathbb{C}, \sigma} \mathbb{C} \to S$.
In this section we define an absolute variation of Hodge structure on $S$ as a collection of variations of Hodge structure $\mathbb{V}^{\sigma}$ on each $\mathrm{Aut}(\mathbb{C} / \mathbb{Q}) $-conjugate of $S$ with certain compatibilities.

\begin{mydef}
An \emph{absolute variation of Hodge structure} is the datum of a $\mathbb{Z}$-variation of Hodge structure $(\mathbb{V}_{\mathbb{Z}}^{\sigma}, \mathcal{V}^{\sigma}, \nabla^{\sigma}, F^{\bullet, \sigma}) $ on $S^{\sigma}$ for each $\sigma \in \mathrm{Aut}(\mathbb{C} / \mathbb{Q})$ together with an isomorphism
\begin{equation}\label{comparisonisomabssp} \iota_{\sigma}: (\mathcal{V}^{\sigma} , \nabla^{\sigma}, F^{\bullet, \sigma}) \cong {\sigma^{-1}}^{*} (\mathcal{V}^{\mathrm{id}}, \nabla^{\mathrm{id}}, F^{\bullet, \mathrm{id}})\end{equation} of the associated filtered algebraic vector bundles with connection.
A \emph{polarization} of the absolute variation of Hodge structure is a polarization $\alpha^{\sigma}$ on $\mathbb{V}_{\mathbb{Z}}^{\sigma}$ for each $\sigma \in \mathrm{Aut}(\mathbb{C} / \mathbb{Q})$ such that the polarizations correspond to each other under the isomorphism $\iota_{\sigma}$.
\end{mydef}

From now on, all absolute variations are assumed to be polarizable.
We will often forget the $\mathbb{Z}$-structure and write $(\mathbb{V}^{\sigma})_{\sigma}$ for the associated collection of $\mathbb{Q}$-variations of Hodge structure, which we also call an absolute variation of Hodge structure. Note however that we require that all our $\mathbb{Q}$-variations admit a $\mathbb{Z}$-structure. We use the notation $(\mathbb{V}, \mathcal{V}, \nabla, F^{\bullet}):= (\mathbb{V}^{\mathrm{id}}, \mathcal{V}^{\mathrm{id}}, \nabla^{\mathrm{id}}, F^{\bullet, \mathrm{id}} )$.

\begin{rem}
Given a $\mathbb{Z}$-variation of Hodge structure $\mathbb{V}_{\mathbb{Z}}$, using Deligne's canonical extension (cf. \cite{Deligneregular}) one can show that the holomorphic vector bundle $(\mathcal{V}^{an}= \mathbb{V}_{\mathbb{Z}} \otimes_{\mathbb{Z}} \mathcal{O}_{S^{an}}, \nabla^{an})$ together with the holomorphic connection defined by $\mathbb{V}$ arises as the analytification of an algebraic vector bundle with regular algebraic connection $(\mathcal{V}, \nabla)$.
Moreover, making use of Schmidt's nilpotent orbit theorem, one can see that the Hodge filtration arises from an algebraic filtration on $\mathcal{V}$ (\cite{Schmid}, 4.13). Therefore the comparison isomorphism (\ref{comparisonisomabssp}) makes sense for general variations of Hodge structure.
\end{rem}

\begin{ex}\label{geometry}
Let $f: X \to S$ be a smooth projective morphism of irreducible smooth quasi-projective algebraic varieties over $\mathbb{C}$.
For any automorphism $\sigma \in \mathrm{Aut}(\mathbb{C}/\mathbb{Q})$ we consider the base change $f^{\sigma}: X^{\sigma} \to S^{\sigma}$.
The collection of variations of Hodge structure $\mathbb{V}^{\sigma} := R_{prim}^kf^{\sigma,an}_{*} \underline{\mathbb{Q}}$ on $S^{\sigma}$ is an absolute variation of Hodge structure on $S$.

\end{ex}

Suppose that $S$ is defined over a subfield $K \subset \mathbb{C}$ and $\mathcal{V}$, $\nabla$ and $F^{\bullet}$ are all defined over $K$.
In this case, an absolute variation of Hodge structure satisfying $\mathbb{V}^{\sigma} = \mathbb{V}$ for all $\sigma \in \mathrm{Aut}(\mathbb{C}/K)$ is called a \emph{$K$-absolute variation}.
For example, if $K$ is a number field, we only need to specify $\mathbb{V}^{\sigma}$ for the finitely many $\sigma \in \mathrm{Gal}(K/\mathbb{Q})$.
If the morphism $f$ is defined over $K$, then Example \ref{geometry} produces a $K$-absolute variation.

We call an absolute variation of Hodge structure on a point an \emph{absolute Hodge structure}. It is a collection of $\mathbb{Q}$-Hodge structures $V^{\sigma}$ together with isomorphisms $\iota_{\sigma}: V^{\sigma}\otimes \mathbb{C} \cong V\otimes \mathbb{C}$ respecting the filtration.
Absolute Hodge structures naturally form a Tannakian category where the morphisms are morphisms of Hodge structures which are dR-absolute Hodge in the sense of the definition in the next section.

\subsection{dR-absolute Hodge cycles}

In this section we recall a slightly weaker version considered in \cite{Voisin} of Deligne's notion of absolute Hodge cycle.
Choose a point $s \in S(\mathbb{C})$. For every tensor $\alpha \in \mathbb{V}_{s}^{\otimes (m,n)}$, the comparison $\mathbb{V}_{s}^{\otimes (m,n)} \otimes \mathbb{C} \cong \mathcal{V}_{s}^{\otimes (m,n)}$ defines a de Rham tensor $\alpha_{dR} \in \mathcal{V}^{\otimes (m,n)}_s$.

When given an absolute variation of Hodge structure $(\mathbb{V}^{\sigma})_{\sigma}$ on $S$, conjugation by any element $\sigma \in \mathrm{Aut}(\mathbb{C} / \mathbb{Q})$ gives a conjugated de Rham tensor $\alpha_{dR}^{\sigma} \in (\mathcal{V}^{\sigma}_{s^{\sigma}})^{\otimes (m,n)}$.

\begin{mydef}[\cite{Voisin}, Definition 0.1]\label{defdRabsoluteHodge}
The tensor $\alpha \in \mathbb{V}_s^{\otimes (m,n)}$ is called \emph{dR-absolute Hodge} if for every $\sigma \in \mathrm{Aut}(\mathbb{C} / \mathbb{Q})$ the conjugate $\alpha_{dR}^{\sigma} \in (\mathcal{V}^{\sigma}_{s^{\sigma}})^{\otimes (m,n)}$ is induced by a Hodge tensor in $(\mathbb{V}^{\sigma}_{s^{\sigma}})^{\otimes (m,n)}$.
\end{mydef}

A powerful tool in the study of absolute Hodge classes is Deligne's Principle B (cf. \cite{Deligne}, Theorem 2.12).
\begin{thm}\label{PrincipleBnonsmooth}
Let $\xi$ be a global section of $\mathbb{V}^{\otimes (m,n)}$ over a finite étale cover $Z'$ of a closed irreducible subvariety $Z\subset S$ such that $\xi_s$ is dR-absolute Hodge for some point $s \in Z'(\mathbb{C})$. Then $\xi_t$ is dR-absolute Hodge for every point $t \in Z'(\mathbb{C})$.
\end{thm}
\begin{proof}
To the global section $\xi$ over $Z'$ we can attach a flat section $\xi_{dR}$ of $\left. \mathcal{V}^{\otimes (m,n)} \right|_{Z'}$.
Conjugation by $\sigma \in \mathrm{Aut}(\mathbb{C} / \mathbb{Q})$ produces a global section $\xi^{\sigma}_{dR}$ of the vector bundle $\left. (\mathcal{V} ^{\sigma})^{\otimes (m,n)} \right|_{{{Z'}^{\sigma}}} $. Since the connection is also algebraic, $\nabla^{\sigma} \xi_{dR}^{\sigma}=0$. Hence $\xi_{dR}^{\sigma}$ is a flat section and applying the Riemann-Hilbert correspondence to the restriction $\left. \xi^{\sigma}_{dR} \right|_{{Z'}^{\sigma, sm}}$ to ${Z'}^{\sigma, sm}$ gives a global section of the $\mathbb{C}$-local system $\left. (\mathbb{V}_{\mathbb{C}}^{\sigma})^{\otimes (m,n)} \right|_{{Z'}^{\sigma, sm}}$ over the smooth locus of ${Z'}^{\sigma}$.
By Lemma \ref{monsm}, this global section uniquely extends to a global section $\xi^{\sigma}$ on ${Z'}^{\sigma}$.
The associated flat section of $\left. (\mathcal{V} ^{\sigma})^{\otimes (m,n)} \right|_{{Z'}^{\sigma}} $ is $\xi^{\sigma}_{dR}$ because this is true over the dense open subset ${Z'}^{\sigma, sm }$. By assumption, $\xi^{\sigma}_{s^{\sigma}}$ is a Hodge cycle, so by Theorem \ref{fixedpartnonsm1} the fiber $\xi^{\sigma}_{t^{\sigma}}$ is a Hodge cycle for every $t \in Z'(\mathbb{C})$.
We conclude that $\xi_t$ is dR-absolute Hodge for every $t \in Z'(\mathbb{C})$.
\end{proof}

Let $\mathcal{AH}_Z \subset \mathbb{V}_s^{\otimes}$ denote the subset of all tensors $v \in \mathbb{V}_s^{\otimes}$ such that the translates of $v$ by parallel transport are dR-absolute Hodge at every point $t \in Z(\mathbb{C})$.

\begin{mydef}We define the generic dR-absolute Mumford-Tate group $G_{Z}^{AH}$ to be the subgroup of $GL(\mathbb{V}_s)$ fixing all tensors in $\mathcal{AH}_Z$.

\end{mydef}
\begin{rem}\label{grouppoint1}
If $Z=\{s\}$ is a point, then we may define $G_Z^{AH}$ as the Tannakian group of the subcategory generated by the absolute Hodge structure $(\mathbb{V}^{\sigma}_{s^{\sigma}})_{\sigma}$ and its dual inside the category of absolute Hodge structures.
To see this, note that since the absolute Hodge structure is polarized the Tannakian group is a (possibly non-connected) reductive group and therefore characterized by the tensors it fixes. Since these are exactly the dR-absolute Hodge tensors, we recover the above definition.
In the geometric situation of Example \ref{geometry}, the group $G_s^{AH}$ is then the motivic Galois group of $H_{prim}^k(X_s)$ in terms of motives for dR-absolute Hodge cycles (cf. \cite{DM}, §6 for the variant which uses the stronger notion of absolute Hodge cycles).
\end{rem}

Since every dR-absolute Hodge cycle is in particular a Hodge cycle we see that we have the inclusion $G_Z \subset G_Z^{AH}$.

We make crucial use of the following Lemma, which is a geometric incarnation of Deligne's Principle B:

\begin{lem}\label{lem2}
Suppose $Z \subset Y$ are closed irreducible subvarieties that satisfy $H_Y \subset G^{AH}_Z$. Then $G^{AH}_Z = G^{AH}_Y$.

\end{lem}
\begin{proof}
We show that $G^{AH}_Y \subset G^{AH}_Z$, the other inclusion being clear.
Let $s \in Z(\mathbb{C})$. The condition $H_Y \subset G^{AH}_Z$ translates into $\mathcal{AH}_Z \subset (\mathbb{V}_s^{\otimes})^{H_Y}$.
Now every element of $v \in \mathcal{AH}_Z$ is fixed by $H_Y$ and is dR-absolute Hodge at the point $s \in Z(\mathbb{C})$. We apply Theorem \ref{PrincipleBnonsmooth} to show that $v$ extends to a global section over a finite étale cover $Y'$ of $Y$ which is dR-absolute Hodge at every point of $Y'$, and so $\mathcal{AH}_Z \subset \mathcal{AH}_Y$.
\end{proof}

\begin{prop}\label{HZisnormalinGAH}
The group $G_Z^{AH}$ is a (possibly non-connected) reductive group.
The algebraic monodromy group $H_Z$ is a normal subgroup of $G_Z^{AH}$.
\end{prop}
\begin{proof}
Again, the argument in the proof of Theorem \ref{PrincipleBnonsmooth} shows that $G_Z^{AH}$ does not change when we replace $Z$ by its smooth locus and we may therefore assume that $Z$ is smooth.
We claim that there exists a point $s \in Z(\mathbb{C})$ with $G_s^{AH}=G_Z^{AH}$. In fact, if we choose any Hodge generic point $s$, then
$H_Z \subset G_s \subset G_s^{AH}$ and the above Lemma \ref{lem2} gives $G_s^{AH}=G_Z^{AH}$.
Now the reductivity follows from Remark \ref{grouppoint1}.
Following the proof of (\cite{Andre}, Theorem 1), in order to show that $H_Z \unlhd G_Z^{AH}$ it is enough to show that $\left(\mathbb{V}_s^{\otimes (m,n)}\right)^{H_Z} \subset \mathbb{V}_s^{\otimes (m,n)} $ carries an action of $G_Z^{AH}=G_s^{AH}$.
This follows from the fact that the collection $\left(\left((\mathbb{V}_{s^{\sigma}}^{\sigma})^{\otimes (m,n)}\right)^{H_{Z^{\sigma}}}\right)_{\sigma}$ forms a sub-absolute Hodge structure of $\left((\mathbb{V}_{s^{\sigma}}^{\sigma})^{\otimes (m,n)}\right)_{\sigma}$ and is therefore preserved by $G_s^{AH}$.
\end{proof}

\section{dR-absolutely special subvarieties}

In this section we introduce the notion of dR-absolutely special subvariety, study its properties and prove that weakly non-factor special subvarieties are dR-absolutely special.

\subsection{Definition and first properties}
Let $(\mathbb{V}^{\sigma})_{\sigma}$ be an absolute variation of Hodge structure on a smooth irreducible quasi-projective complex algebraic variety $S$ and $Z \subset S$ a closed irreducible subvariety.
Since we do not know whether $G_Z = G_Z^{AH}$, it is natural to define a notion of \emph{dR-absolutely special} subvariety simply by replacing $G_Z$ in Definition \ref{definitionspecial1} by $G_Z^{AH}$.

\begin{mydef}\label{absspdef}
A closed irreducible algebraic subvariety $Z \subset S$ is called \emph{dR-absolutely special} if it is maximal among the closed irreducible algebraic subvarieties of $S$ with generic dR-absolute Mumford-Tate group $G_Z^{AH}$.
\end{mydef}

We mention a few properties of dR-absolutely special subvarieties:

\begin{prop} \label{absspecial1}
Let $Z$ be a dR-absolutely special subvariety.
\begin{enumerate}[(i)]
\item \label{both1}
$Z$ is special.
\item \label{conj1}
$Z^{\sigma}\subset S^{\sigma}$ is dR-absolutely special for all $\sigma \in \mathrm{Aut}(\mathbb{C} / \mathbb{Q})$.
\item \label{three1}
If $S$ is defined over a finite extension $K$ of ${\mathbb{Q}}$ and the variation is $K$-absolute, then $Z$ is defined over $\bar{\mathbb{Q}}$ and all its $\mathrm{Gal}(\bar{\mathbb{Q}}/K)$-conjugates are special.
\end{enumerate}
\end{prop}
\begin{proof}
Suppose that $Y \supset Z$ is closed irreducible such that $G_Y = G_Z$. It follows that $H_Y \subset G_Z \subset G^{AH}_{Z}$, and therefore $G_Y^{AH} = G_Z^{AH}$ by Lemma \ref{lem2}.
We conclude that $Z=Y$ because $Z$ is dR-absolutely special. This shows that $Z$ is special.
Assertion (\ref{conj1}) follows from the fact that we have a natural isomorphism $G^{AH}_Z \otimes \mathbb{C} \cong G^{AH}_{Z^{\sigma}} \otimes \mathbb{C}$.
By the fact that there are only countably many special subvarieties of $S$ for $\mathbb{V}$, this implies (\ref{three1}).
\end{proof}

Note that dR-absolutely special subvarieties have the good arithmetic properties that conjecturally hold for special subvarieties.

Deligne's conjecture that Hodge classes are dR-absolute Hodge states that we have an equality of groups $G_Z = G_Z^{AH}$.
For dR-absolutely special subvarieties we can at least show that there does not exist any closed irreducible subvariety $Y \supset Z$ whose generic Mumford-Tate group contradicts this equality.

\begin{prop}\label{nogroup}
Let $Z$ be dR-absolutely special.
Then there does not exists a closed irreducible subvariety $Y \supset Z$ with the property
$$G_Z \subsetneq G_Y \subset G^{AH}_Z.$$
\end{prop}
\begin{proof}
Let $Y \supset Z$ be a closed irreducible subvariety with $G_Y \subset G^{AH}_Z$. It follows that $H_Y \subset G^{AH}_{Z}$, and therefore Lemma \ref{lem2} implies $G_Y^{AH} = G_Z^{AH}$.
We conclude that $Z=Y$ because $Z$ is dR-absolutely special.
\end{proof}

For the next Proposition, suppose that $S$ is defined over $\bar{\mathbb{Q}}$ and the variation is a $\bar{\mathbb{Q}}$-absolute variation.

\begin{prop}\label{Qbarclosure}
Let $Z$ be a closed irreducible subvariety and $Y:= \overline{Z}^{Zar, \bar{\mathbb{Q}}}$
the smallest $\bar{\mathbb{Q}}$-subvariety containing $Z$.
Then $H_Z$ is a normal subgroup of $H_Y$, $G_Y$ and $G_Y^{AH}$. \end{prop}
\begin{proof}
Let $W \supset Z $ be a maximal subvariety with $G_W^{AH} = G_Z^{AH}$. Then $W$ is dR-absolutely special and therefore defined over $\bar{\mathbb{Q}}$ by Proposition \ref{absspecial1}. In particular, $Y \subset W$ and thus $G_Y^{AH}= G_Z^{AH}$.
It follows from Proposition \ref{HZisnormalinGAH} that $H_Z$ is a normal subgroup of $G_Y^{AH}$.
Now the inclusions $H_Z \subset H_Y \subset G_Y \subset G_Y^{AH}$ show that $H_Z$ is also normal in $H_Y$ and $G_Y$.
\end{proof}

\begin{rem}
The normality of $H_Z$ in $H_Y$ was observed in the course of the proof of
(\cite{KOU}, Proposition 3.2).
\end{rem}

\subsection{dR-absolutely special subvarieties and period maps}
Let $Z \subset S$ be a closed irreducible subvariety.
Denote by $(G_Z^{AH}, \mathcal{D}_{Z}^{AH})$ and $(G_S^{AH}, \mathcal{D}_S^{AH})$ the Hodge data defined by $G_Z^{AH}$ and $G_S^{AH}$.
We want to describe the dR-absolutely special subvarieties of $S$ in terms of the period map
$$\Phi: S^{an} \to \Gamma \backslash \mathcal{D}_{S}^{AH}.$$
The inclusion $G_Z^{AH} \subset G_S^{AH}$ induces a morphism of Hodge varieties $$\iota:\Gamma_{Z}^{ AH} \backslash \mathcal{D}^{AH}_{Z} \to \Gamma \backslash \mathcal{D}_{S}^{AH}.$$
Here $\Gamma_{Z}^{AH} := \Gamma \cap G_Z^{AH}(\mathbb{Q})$.
\begin{prop}\label{absspperiod}
The subvariety $Z$ is dR-absolutely special if and only if $Z$ is a (complex analytic) irreducible component of $\Phi^{-1}(\iota(\Gamma_{Z}^{AH} \backslash \mathcal{D}^{AH}_{Z}))$.
\end{prop}
\begin{proof}
Clearly the restriction of the period map $\Phi$ to any subvariety $Y \supset Z$ with $G_Y \subset G_Y^{AH}=G_Z^{AH}$ factors through $\iota(\Gamma_{Z}^{AH} \backslash \mathcal{D}^{AH}_{Z})$.
Conversely, let $W$ be a complex analytic irreducible component of $\Phi^{-1}(\iota(\Gamma_{Z}^{AH} \backslash \mathcal{D}^{AH}_{Z}))$ containing $Z$.
It follows from (\cite{BKT}, Theorem 1.6) that $W$ is algebraic, and it satisfies $G_W \subset G_Z^{AH}$. Applying Lemma \ref{lem2}, we see that $G_W^{AH} = G_Z^{AH}$.
As $Z$ is dR-absolutely special this gives $Z=W$.
\end{proof}

Note that $\Phi^{-1}(\iota(\Gamma_{Z}^{AH} \backslash \mathcal{D}^{AH}_{Z}))$ may have several (but finitely many) irreducible components and therefore $Z$ is not necessarily contained in a unique smallest dR-absolutely special subvariety.
Still, $Z$ is contained in a unique irreducible component of an intersection of irreducible components of $\Phi^{-1}(\iota(\Gamma_{Z}^{AH} \backslash \mathcal{D}^{AH}_{Z}))$.
Components of this form will be called
\emph{dR-absolutely special intersections}.

\begin{cor}
Every subvariety $Z$ is contained in a unique smallest dR-absolutely special intersection, called the \emph{dR-absolutely special closure} of $Z$.
\end{cor}

\subsection{Weakly non-factor subvarieties}

Weakly non-factor subvarieties are a class of subvarieties that satisfy a certain monodromy condition introduced in (\cite{KOU}, Definition 1.10).
Their significance for the arithmetic properties of special subvarieties was already established in \cite{KOU}.
In this section, we prove a slight strengthening of their result: any special subvariety which is weakly non-factor is in fact dR-absolutely special.

\begin{mydef}[\cite{KOU}, Definition 1.10]
A closed irreducible subvariety $Z \subset S$ is called \emph{weakly non-factor} if it is not contained in a closed irreducible $Y \subset S$ such that $H_{Z}$ is a strict normal subgroup of $H_{Y}$.
\end{mydef}

\begin{thm}\label{wswnf1}
Assume $Z \subset S$ is weakly special and weakly non-factor. Then $Z$ is dR-absolutely special.
\end{thm}

\begin{proof}
Assume $Y \supset Z$ is a closed irreducible subvariety such that $G^{AH}_{Y} = G^{AH}_{Z}$. By Proposition \ref{HZisnormalinGAH}, the groups $H_Z$ and $H_{Y}$ are both normal in $G_Y^{AH}=G_Z^{AH}$, and therefore $H_{Z}$ is a normal subgroup of $H_{Y}$. The fact that $Z$ is weakly non-factor now implies $H_Z=H_Y$. We conclude that $Z=Y$ since $Z$ is weakly special. Thus $Z$ is dR-absolutely special.
\end{proof}
As a consequence we see that Conjecture \ref{conjabssp} holds true for weakly non-factor subvarieties.
\begin{cor}\label{swnf}
Assume $Z \subset S$ is a special subvariety which is weakly non-factor. Then $Z$ is dR-absolutely special.
\end{cor}

Suppose $S$ is defined over $\bar{\mathbb{Q}}$ and the variation $(\mathbb{V}^{\sigma})_{\sigma}$ is $\bar{\mathbb{Q}}$-absolute. In this case, Corollary \ref{swnf} and Proposition \ref{absspecial1} imply that a special, weakly non-factor subvariety is defined over $\bar{\mathbb{Q}}$ and all its Galois conjugates are special. This was already proven in \cite{KOU}.

Write $G_S^{ad} = G_1 \times G_2 \times... \times G_n$ as a product of simple factors. This gives rise to a product decomposition of the Hodge variety
\begin{equation}\label{decompositionHodgevariety} \Gamma \backslash \mathcal{D} =\Gamma_1 \backslash \mathcal{D}_1 \times \Gamma_2 \backslash \mathcal{D}_2 \times...\times \Gamma_n \backslash \mathcal{D}_n. \end{equation}
\begin{cor}\label{maxspprpos}
Let $Z \subset S$ be a maximal strict special subvariety with the property that the projection of $\Phi(Z^{an})$ to each simple factor $\Gamma_i \backslash \mathcal{D}_i$ is positive dimensional. Then $Z$ is dR-absolutely special.
\end{cor}
\begin{proof}
We prove that $Z$ is weakly non-factor to apply Theorem \ref{wswnf1}.
Let $Y\supset Z$ be such that $H_Z$ is a strict normal subgroup of $H_Y$. Then as $Z$ is a maximal strict special subvariety we have the equality $G_Y = G_S$. It follows that $H_Y \unlhd H_S$ and the fact that the projection is positive dimensional on each simple factor implies that $H_Y= H_S$. Thus $H_Z$ is a proper normal subgroup of $H_S$. This contradicts the fact that the projection of $Z$ to each simple factor is positive dimensional.
\end{proof}

If $G_S^{ad}$ is simple, this reduces to the following example taken from (\cite{KOU}, Corollary 1.13).
An irreducible subvariety $Z \subset S$ is called \emph{of positive period dimension} if $\Phi(Z^{an})$ is not a point, or equivalently, if the algebraic monodromy group $H_Z$ is non-trivial.

\begin{cor}\label{exampleKOU}
Suppose that $G_{S}^{ad}$ is simple.
Let $Z \subset S$ be a strict special subvariety which is of positive period dimension, and maximal for these properties.
Then $Z$ is dR-absolutely special.
\end{cor}
\begin{rem}
Instead of the adjoint group $G_S^{ad}$ one can also assume that the derived group $G_S^{der}$ is simple. Since $G_S^{der}$ is an extension of $G_S^{ad}$ by a finite group, the simpleness of $G_S^{der}$ is equivalent to that of $G_S^{ad}$. As in \cite{KOU}, we use the adjoint group because it appears naturally in the decomposition (\ref{decompositionHodgevariety}).
\end{rem}

Denote by $HL_{pos}$ the Hodge locus of positive period dimension, which is defined as the union of all strict special subvarieties of $S$ which are of positive period dimension (\cite{KO}, Definition 1.4).
Then the Corollary shows that if $G_S^{ad}$ is simple, then $HL_{pos}$ is a countable union of dR-absolutely special subvarieties.
In particular, $HL_{pos}$ is cut out by dR-absolute Hodge cycles.

\section{Applications}

We give applications to several arithmetic questions in Hodge theory.
As before, we let $(\mathbb{V}^{\sigma})_{\sigma}$ be an absolute variation of Hodge structure on a smooth irreducible quasi-projective complex algebraic variety $S$.

\subsection{dR-absolutely special subvarieties and typical intersections}

The property that all generic Hodge cycles on a special subvariety $Z$ are dR-absolute Hodge (i.e. $G_Z = G_Z^{AH}$) implies that $Z$ is dR-absolutely special, but the converse does not hold.
For instance, note that for the examples of dR-absolutely special subvarieties given by Corollary \ref{exampleKOU}, we cannot determine the group $G_Z^{AH}$.
In this section we will see that for a dR-absolutely special subvariety we can give an upper bound for the "difference" between $G_Z$ and $G_Z^{AH}$ in terms of the Hodge codimension.
The strength of this bound depends on how close $Z$ is to being a \emph{typical intersection}.

\begin{mydef}[\cite{BKU}, Definition 4.1]
Let $Z\subset S$ be a closed irreducible subvariety. The \emph{Hodge codimension} of $Z$ is defined to be
$$ Hcd(Z) := \dim \Gamma_Z \backslash \mathcal{D}_Z - \dim \Phi(Z^{an}),$$
where $\Gamma_Z \backslash \mathcal{D}_Z $ is the Hodge variety for the generic Mumford-Tate group $G_Z$ of $Z$.
\end{mydef}

\begin{mydef}[\cite{BKU}, Definition 4.2]\label{defatypicalintersectionabssp}
A closed irreducible subvariety $Z \subset S$ is called \emph{atypical} if
$$ Hcd(S) > Hcd(Z). $$
Otherwise, $Z$ is called \emph{typical}.
\end{mydef}

Note that if $Z$ is special, the non-strict inequality $Hcd(S) \ge Hcd(Z)$ always holds.
\begin{rem}
Strictly speaking, one has to modify Definition \ref{defatypicalintersectionabssp} slightly to take possible singularities of the image $\Phi(S^{an})$ of the period mapping into account, compare (\cite{BKU}, Definition 4.2).
\end{rem}

\begin{mydef}\label{defect}
For a closed irreducible subvariety $Z$, we define the \emph{absolute Hodge defect} to be $$\delta_{AH}(Z):= \dim \Gamma_Z^{AH} \backslash \mathcal{D}_Z^{AH} - \dim \Gamma_Z \backslash \mathcal{D}_Z.$$
\end{mydef}

\begin{thm}\label{bound}
Let $Z$ be a dR-absolutely special subvariety and assume that $G_S^{AH}=G_S$, i.e. all generic Hodge cycles on $S$ are dR-absolute Hodge.
Then $$\delta_{AH}(Z) \le Hcd(S)-Hcd(Z).$$
\end{thm}
\begin{proof}
It follows the description of dR-absolutely special subvarieties in Proposition \ref{absspperiod} that
$$ \dim \Gamma\backslash \mathcal{D}_S -\dim \Phi(S^{an}) \ge \dim \Gamma_Z^{AH} \backslash \mathcal{D}_Z^{AH} - \dim \Phi(Z^{an}).$$
We get

\begin{eqnarray*}
\delta_{AH}(Z) &= & \dim \Gamma_Z^{AH} \backslash \mathcal{D}_Z^{AH} - \dim \Gamma_Z \backslash \mathcal{D}_Z \\
& \le & (\dim \Gamma\backslash \mathcal{D}_S -\dim \Phi(S^{an})) - (\dim \Gamma_Z \backslash \mathcal{D}_Z - \dim \Phi(Z^{an})) \\ & = & Hcd(S)-Hcd(Z).
\end{eqnarray*}

\end{proof}

\begin{rem} In general, without assuming that all Hodge cycles on $S$ are dR-absolute Hodge, we get the inequality $$\delta_{AH}(Z) - \delta_{AH}(S) \le Hcd(S)-Hcd(Z).$$
\end{rem}

\begin{cor}
Assume that $G_S =G_S^{AH}$.
Let $Z$ be a dR-absolutely special subvariety which is a typical intersection.
Then $\delta_{AH}(Z)= 0$.
\end{cor}

If $S$ is a Shimura variety, then every special subvariety is a typical intersection. In general, the question whether special subvarieties are typical or atypical is closely related to the \emph{level} of the variation of Hodge structure as defined in \cite{BKU}.
For every Hodge generic point $x \in S(\mathbb{C})$, the representation $ \mathbb{S} \to G_{S, \mathbb{R}}$ of the Deligne torus defining the Hodge structure at the point $x$ induces a $\mathbb{Q}$-Hodge structure of weight zero on the Lie algebra $\mathfrak{g}_S$, and the adjoint Lie algebra $\mathfrak{g}_S^{ad}$ can be viewed as a sub-Hodge structure by identifying it with the derived Lie algebra $\mathfrak{g}_S^{der}:= [\mathfrak{g}_S, \mathfrak{g}_S]$. There is a compatibility between the Hodge structure and the Lie algebra structure, and $\mathfrak{g}_S^{ad}$ is called a $\mathbb{Q}$-Hodge Lie algebra in (\cite{BKU}, Definition 3.8).

\begin{mydef}[\cite{BKU}, Definition 3.12 and Definition 3.13] \textnormal{ }
\begin{enumerate}[(i)]
\item
The \emph{level} of an irreducible real Hodge structure $V$ of weight zero is the largest integer $k$ such that $V^{k,-k} \not= 0$ in the decomposition $V_{\mathbb{C}} = \oplus_{k \in \mathbb{Z}} V^{k,-k}$.
We define the level of an irreducible $\mathbb{Q}$-Hodge structure $V$ of weight zero as the maximum of the levels of the irreducible factors of $V_{\mathbb{R}}$, and the level of a $\mathbb{Q}$-Hodge structure of weight zero as the minimum of the levels of its irreducible $\mathbb{Q}$-factors.
\item
The \emph{level} of the variation of Hodge structure $\mathbb{V}$ is defined to be the level of the $\mathbb{Q}$-Hodge structure $\mathfrak{g}_S^{ad}$.
\end{enumerate}
\end{mydef}
One can show that the level is independent of the choice of the Hodge generic point $x$ (\cite{BKU}, 3.6).

For Shimura varieties, the variation $\mathbb{V}$ has level one.
For variations of Hodge structure of level $\ge 3$ however, it was proved in (\cite{BKU}, Theorem 2.3) that all strict special subvarieties are atypical.

\subsection{Maximal atypical special subvarieties}

Under suitable assumptions, maximal atypical special subvarieties of positive period dimension are dR-absolutely special.

\begin{thm}\label{maxatypabssp}
Suppose $G_S^{ad}$ is simple and the level of $\mathbb{V}$ is greater or equal to two. Then any maximal atypical special subvariety of positive period dimension is dR-absolutely special.
\end{thm}
\begin{proof}
If the level of $\mathbb{V}$ is greater or equal to three, it follows from (\cite{BKU}, Theorem 2.3) that all strict special subvarieties are atypical. Hence in this case the statement follows from Corollary \ref{exampleKOU}.
It remains to handle the case of level two. Suppose $Z \subset S$ is a maximal atypical special subvariety of positive period dimension.
If $Y \supsetneq Z$ is a dR-absolutely special subvariety with $G_Y^{AH} = G_Z^{AH}$, then $Y$ is a typical special subvariety. By (\cite{BKU}, Proposition 2.4), the adjoint generic Mumford-Tate group $G_Y^{ad}$ is simple, and thus also the normal subgroup $H_Y$.
Since $H_Z \unlhd G_Z^{AH} = G_{Y}^{AH}$, the group $H_Z$ is a normal subgroup of $H_Y$. Now the fact that $Z$ is of positive period dimension implies $H_Z = H_Y$, which is a contradiction because $Z$ is weakly special.
\end{proof}

\begin{rem}
In the case that $\mathbb{V}$ is of level one, the target of the period map is a Shimura variety, so using Deligne's result (\cite{Deligne}, Theorem 2.11) that Hodge classes on abelian varieties are absolute Hodge one can prove in many cases that special subvarieties are dR-absolutely special.
\end{rem}

\subsection{$\bar{\mathbb{Q}}$-bialgebraic subvarieties}

Let $S$ be an irreducible smooth quasi-projective complex algebraic variety.

\begin{mydef}[\cite{KUY}, Definition 4.1]
A \emph{bialgebraic structure} on $S$ is a pair
$$ \left(D: \tilde{S} \to \check{X}^{an}, h: \pi_1(S^{an}) \to \mathrm{Aut}(\check{X}) \right)$$
where
\begin{itemize}
\item $\pi: \tilde{S} \to S^{an} $ denotes the universal cover of $S^{an}$,
\item $\check{X}$ is a complex algebraic variety,
\item $h: \pi_1(S^{an}) \to \mathrm{Aut}(\check{X})$ is a group homomorphism to the algebraic automorphisms of $\check{X}$,
\item $D$ is a holomorphic map which is $h$-equivariant.
\end{itemize}\end{mydef}

For bialgebraic structures we define a notion of bialgebraic subvarieties.

\begin{mydef}[\cite{KUY}, Definition 4.2 and 4.3]
\textnormal{ }
\begin{enumerate}[(i)]
\item
A closed irreducible analytic subvariety $W \subset \tilde{S}$ is called an \emph{irreducible algebraic subvariety} of $\tilde{S}$ if $W$ is an analytic irreducible component of $D^{-1}(\bar{W})$, where $\bar{W}$ is a closed algebraic subvariety of $\check{X}$.
\item
A closed irreducible algebraic subvariety $Z\subset S$ is called \emph{bialgebraic} if one (equiv. any) irreducible analytic component of $\pi^{-1}(Z)$ is an algebraic subvariety of $\tilde{S}$ in the above sense.
\end{enumerate}\end{mydef}

Consider a variation of Hodge structure $\mathbb{V}$ on $S$.
If $\tilde{\Phi}: \tilde{S} \to \mathcal{D}_S$ denotes the lifted period map and $\iota: \mathcal{D}_S \subset \check{\mathcal{D}}_S$ the embedding of $\mathcal{D}_S$ into its compact dual, let $D= \iota \circ \tilde{\Phi}: \tilde{S} \to \check{\mathcal{D}}_S^{an}$ denote the composition.
Then $$\left(D: \tilde{S} \to \check{\mathcal{D}}_S^{an}, h: \pi_1(S^{an}) \to \mathrm{Aut}(\check{\mathcal{D}}_S) \right)$$ is a bialgebraic structure. Here $h$ is given by the natural map $\pi_1(S^{an}) \to G_S(\mathbb{C})$.

\begin{prop}[\cite{Klingleratypical}, Proposition 7.4]\label{bialgws}
The bialgebraic subvarieties of $S$ for this bialgebraic structure are exactly the weakly special subvarieties.
\end{prop}

In order to detect special subvarieties by the bialgebraic formalism, we need a more refined definition of bialgebraicity that takes arithmetic properties into account.

\begin{mydef}[\cite{KUY}, 4.11]
A bialgebraic structure $$\left(D: \tilde{S} \to \check{X}^{an}, h: \pi_1(S^{an}) \to \mathrm{Aut}(\check{X}) \right)$$ is called \emph{$\bar{\mathbb{Q}}$-bialgebraic} if the following holds:
\begin{enumerate}[(i)]
\item The base variety $S$ is defined over $\bar{\mathbb{Q}}$.
\item The variety $\check{X}$ is defined over $\bar{\mathbb{Q}}$ and the homomorphism $h$ takes values in the automorphisms of $\check{X}$ over $\bar{\mathbb{Q}}$.
\end{enumerate}

\end{mydef}

\begin{mydef}[\cite{KUY}] \label{defbialg}
A closed irreducible subvariety $Z \subset S$ is called \emph{$\bar{\mathbb{Q}}$-bialgebraic} if it is defined over $\bar{\mathbb{Q}}$ and an (resp. any) analytic irreducible component of $\pi^{-1}(Z)$ is an analytic irreducible component of $D^{-1}(\bar{W})$ for a $\bar{\mathbb{Q}}$-subvariety $\bar{W}$ of $\check{X}$.
\end{mydef}

Let $S$ be defined over $\bar{\mathbb{Q}}$ and $(\mathbb{V}^{\sigma})_{\sigma}$ a $\bar{\mathbb{Q}}$-absolute variation on $S$.
Then, as the compact dual $\check{\mathcal{D}}_S= G_{S,\mathbb{C}} / P$ can be defined over $\bar{\mathbb{Q}}$, the period map $\tilde{\Phi}: \tilde{S} \to \check{\mathcal{D}}_S$ defines a $\bar{\mathbb{Q}}$-bialgebraic structure.

\begin{prop}\label{biabssp}
Any dR-absolutely special subvariety is $\bar{\mathbb{Q}}$-bialgebraic.
\end{prop}
\begin{proof}
If $Z \subset S$ is dR-absolutely special, then $Z$ is defined over $\bar{\mathbb{Q}}$. Moreover, since $Z$ is special, $\pi^{-1}(Z)$ will be a union of irreducible components of the preimage of the compact dual $\check{\mathcal{D}}_Z=G_{Z, \mathbb{C}}/P_{Z} \subset \check{\mathcal{D}}_S$ of $\mathcal{D}_Z$. As this inclusion is induced by the inclusion $G_Z \subset G_S$ of algebraic groups over $\mathbb{Q}$ and the cocharacter defining the parabolic subgroup $P_Z$ can be defined over $\bar{\mathbb{Q}}$ (cf. \cite{MilneIntro}, Lemma 12.1), the subvariety $\check{\mathcal{D}}_Z$ is defined over $\bar{\mathbb{Q}}$.
\end{proof}

As a generalization of (\cite{UY}, Theorem 1.4) one conjectures the following:

\begin{conj}\label{absspbi}
The $\bar{\mathbb{Q}}$-bialgebraic subvarieties of $S$ are exactly the dR-absolutely special subvarieties.
\end{conj}

We prove this conjecture for maximal subvarieties of positive period dimension when the adjoint group of $G_S$ is simple.

\begin{thm}\label{bialgmax}
Suppose $G_S^{ad}$ is simple.
Let $Z \subset S$ be a strict maximal $\bar{\mathbb{Q}}$-bialgebraic subvariety of positive period dimension. Then $Z$ is a dR-absolutely special subvariety (and in particular a special subvariety).
\end{thm}
\begin{proof}
Let $Y \supset Z$ be a dR-absolutely special subvariety such that $G_Y^{AH}=G_Z^{AH}$. By Proposition \ref{biabssp}, the variety $Y$ is $\bar{\mathbb{Q}}$-bialgebraic. By the maximality, either $Y= Z$ or $Y=S$. We have to exclude the latter case.
Suppose $Y=S$, then $H_Z \unlhd G_Z^{AH}=G_S^{AH}$. It follows that $H_Z \unlhd H_S$. As $H_S$ is simple, this forces $H_Z =H_S$ since $Z$ is of positive period dimension. But $Z$ is weakly special by Proposition \ref{bialgws}, so this would mean that $Z=S$ which is a contradiction to the fact that $Z$ is a strict subvariety.
\end{proof}

\subsection{Reduction to the case of points}

Corollary \ref{maxspprpos} suggests that in some sense special points are the hardest case for Conjectures \ref{conjdel1} and \ref{conjabssp}.
In this section we prove that both can in fact be reduced to the case of special points.

Let $(\mathbb{V}^{\sigma})_{\sigma}$ be a $\bar{\mathbb{Q}}$-absolute variation of Hodge structure on $S$ of geometric origin.
For a closed irreducible subvariety $Z\subset S$ we denote by $Y$ the union of all dR-absolutely special subvarieties of $S$ containing $Z$ with generic dR-absolute Mumford-Tate group $G_Z^{AH}$. These are only finitely many by Proposition \ref{absspperiod}.
We claim that $Z^{\sigma}$ is special in $S^{\sigma}$ for $\mathbb{V}^{\sigma}$ if and only if $Z^{\sigma}$ is special in $Y^{\sigma}$ for $\left. \mathbb{V}^{\sigma} \right|_{Y^{\sigma}} $.
Indeed, if $W \supset Z^{\sigma}$ is such that $G_W=G_{Z^{\sigma}}$, then Lemma \ref{lem2} shows that $G_W^{AH}=G_{Z^{\sigma}}^{AH}$ and hence $W$ is contained in $Y^{\sigma}$.
Replacing $S$ by $Y$, we may therefore assume that $G_S^{AH}=G_Z^{AH}$.
It follows that $H_Z$ is a normal subgroup of $H_S, G_S$ and $G_S^{AH}$.
As $H_Z$ is the stabilizer of finitely many tensors in $\mathbb{V}_s^{\otimes}$ for some $s \in Z(\mathbb{C})$, there exists a finite collection of integers $(a_i, b_i)_{i \le n}$ such the variation of Hodge structure $\mathbb{V}':= \bigoplus_{i \le n}(\mathbb{V}^{\otimes(a_i, b_i)})^{H_Z}$
has algebraic monodromy $H_S/H_Z$, generic Mumford-Tate group $G_S/H_Z$ and generic dR-absolute Mumford-Tate group $G_S^{AH}/H_Z$.
Indeed, for this to be the case we need to ensure that $G_S^{AH}/H_Z$ acts faithfully on the fiber $\mathbb{V}'_s$, which is guaranteed if $\mathbb{V}'_s$ contains the finitely many tensors defining $H_Z$.

The collection $({\mathbb{V}'}^{ \sigma})_{\sigma}$ defined by $${\mathbb{V}'}^{ \sigma}:= \bigoplus_{i \le n}\left((\mathbb{V}^{\sigma})^{\otimes(a_i,b_i)}\right)^{H_{Z^{\sigma}}}$$ forms an absolute variation on $S$.

After replacing $S$ by a desingularization $\tilde{S} \to S$, we may assume that $S$ is smooth.
The period map attached to $\mathbb{V}'$ is
$$ \Phi': S^{an} \to \Gamma'\backslash\mathcal{D}' ,$$
where $\Gamma' \backslash \mathcal{D}' $ is a Hodge variety associated with the quotient $G_S / H_Z$.
Let $\bar{S}$ be a smooth compactification of $S$ over $\bar{\mathbb{Q}}$ by a normal crossings divisor.
Denoting by $\tilde{S} \subset \bar{S}$ the subset where the local monodromy is finite, by (\cite{periodbook}, Corollary 13.7.6)
the period map $\Phi'$ extends to a proper map $$\tilde{\Phi}': \tilde{S}^{an} \to \Gamma' \backslash \mathcal{D}'.$$
Hence again replacing $S$ by $\tilde{S}$ we may assume that the period map $\Phi'$ is proper.
By \cite{BBT} there is a factorization $ \Phi' = \Psi \circ f$, where
$f: S \to B$ is a proper surjective map of algebraic varieties with connected fibers
and $\Psi: B^{an} \to \Gamma' \backslash \mathcal{D}'$ is a quasi-finite period map.

\begin{prop}\label{kernel}
For every $\sigma \in \mathrm{Aut}(\mathbb{C} / \mathbb{Q})$, the period map $$\Phi'_{\sigma}: S^{\sigma, an} \to \Gamma'_{\sigma} \backslash \mathcal{D}'_{\sigma}$$ corresponding to the variation of Hodge structure ${\mathbb{V}'}^{\sigma}$
also factors as $$\Phi'_{\sigma} = \Psi_{\sigma}\circ f^{\sigma},$$ where $f^{\sigma}: S^{\sigma} \to B^{\sigma}$ is the $\sigma$-conjugate of $f: S \to B$ and $\Psi_{\sigma}: B^{\sigma,an} \to \Gamma'_{\sigma} \backslash \mathcal{D}'_{\sigma} $ is a quasi-finite period map.
In particular, there exists a variation of Hodge structure $ {\mathbb{V}'}^{ \sigma}_B$ on $B^{\sigma}$ such that ${\mathbb{V}'}^{ \sigma} = f^{\sigma, *} {\mathbb{V}'}^{ \sigma}_B$.
\end{prop}
\begin{proof}
The first map $f: S \to B$ in the Stein factorization of $\Phi'$ is characterized by the fact that the kernel $ S \times_B S$ is the connected component of the diagonal of the kernel the period map $\Phi': S^{an} \to \Gamma' \backslash \mathcal{D}'$.
Thus it suffices to prove that for every $\sigma$, the kernel $S^{\sigma} \times_{B^{\sigma}} S^{\sigma}$ equals the connected component of the diagonal of the kernel of the period map $\Phi'_{\sigma}:S^{\sigma, an} \to \Gamma'_{\sigma} \backslash \mathcal{D}'_{\sigma}$.
This follows from the fact that the Hodge cycle defining this kernel is dR-absolute Hodge (as it is the identity, and hence dR-absolute Hodge on the diagonal).
\end{proof}

\begin{prop}
The map $f: S \to B$ can be defined over $\bar{\mathbb{Q}}$, and the collection $({\mathbb{V}'}^{ \sigma}_B)_{\sigma}$ forms a $\bar{\mathbb{Q}}$-motivic variation of Hodge structure on $B$.
\end{prop}
\begin{proof}
We claim that ${\mathcal{V}'}^{ \sigma}_B = f^{\sigma}_{*} {\mathcal{V}'}^{ \sigma}$, where ${\mathcal{V}'}^{ \sigma}_B$ and ${\mathcal{V}'}^{ \sigma} $ are the associated algebraic vector bundles.
Following the argument in (\cite{KOU}, Lemma 3.4), the projection formula gives
\begin{equation}\label{projectionformulaabssp} f^{\sigma}_{*} {\mathcal{V}'}^{ \sigma} = f^{\sigma}_{*}(f^{\sigma, *} {\mathcal{V}'}^{ \sigma}_B \otimes_{\mathcal{O}_{S^{\sigma}}} \mathcal{O}_{S^{\sigma}}) = {\mathcal{V}'}^{ \sigma}_B \otimes_{\mathcal{O}_{B^{\sigma}}} f_{*}^{\sigma} \mathcal{O}_{S^{\sigma}} = {\mathcal{V}'}^{ \sigma}_B, \end{equation}
and similarly for the filtration and the connection.
Here we use that by the Stein factorization $f_{*}^{\sigma} \mathcal{O}_{S^{\sigma}} = \mathcal{O}_{B^{\sigma}}$.
Now the comparison isomorphisms for the absolute variation of Hodge structure $({\mathbb{V}'}^{ \sigma}_B)_{\sigma}$ follow from those for ${\mathbb{V}'}^{ \sigma} $ by applying $f^{\sigma}_{*}$.
We prove that the absolute variation in question is a $\bar{\mathbb{Q}}$-motivic variation.
By construction, there is a dense open subvariety $U \subset S$ such that for each point $x \in U(\mathbb{C})$, the collection $(\mathbb{V}_{x^{\sigma}}^{ \sigma})_{\sigma}$ comes from a motive for dR-absolute Hodge cycles.
Since $f$ is proper, the image of $S \setminus U$ under $f$ is closed subvariety of $B$, and we claim that for $b =f(x)\in B(\mathbb{C})$ chosen away from this image, the collection $({\mathbb{V}'}^{ \sigma}_{B,b^{\sigma}})_{\sigma}$ comes from a motive for dR-absolute Hodge cycles.
As ${\mathbb{V}'} = f^{*} \mathbb{V}'_B$, we have $$ \mathbb{V}'_{B,b} = \mathbb{V}'_x = \bigoplus_{i \le n}\left(\mathbb{V}_x^{\otimes(a_i, b_i)}\right)^{H_Z}.$$
From the normality of $H_Z$ in $G_S^{AH}$ we see that the variation of Hodge structure $\mathbb{V}'$ on $S$ is the kernel of an idempotent operator on $\bigoplus_{i \le n}\mathbb{V}^{\otimes(a_i, b_i)}$ which is dR-absolute Hodge.
In particular, if $(\mathbb{V}_{x^{\sigma}}^{ \sigma})_{\sigma}$ comes from a motive for dR-absolute Hodge cycles, then $({\mathbb{V}'}^{ \sigma}_{B,b^{\sigma}})_{\sigma} = ({\mathbb{V}'}^{\sigma}_{x^{\sigma}})_{\sigma} $ is the realization of a motive lying in the Tannakian subcategory generated by this motive.
We claim that the filtered algebraic vector bundle with connection $(\mathcal{V}', \nabla', F^{',\bullet})$ on $S$ is defined over $\bar{\mathbb{Q}}$.
Indeed, the de Rham component of the idempotent dR-absolute Hodge operator on $\bigoplus_{i \le n}\mathbb{V}^{\otimes(a_i, b_i)}$ defining $\mathbb{V}'$ is defined over $\bar{\mathbb{Q}}$, and its kernel is $\mathcal{V}'$.
By the remark following (\cite{BBT}, Theorem 1.1), the morphism $f: S \to B$ is defined over $\bar{\mathbb{Q}}$. Using (\ref{projectionformulaabssp}) we conclude that $({\mathbb{V}'}^{ \sigma}_{B})_{\sigma}$ is a $\bar{\mathbb{Q}}$-motivic variation.
\end{proof}

For every closed subvariety $W$ of $S$, the image $W':=f(W)$ is a closed subvariety of $B$. Then $W$ is a special subvariety of $S$ for $\mathbb{V}'$ if and only if $W'$ is a special subvariety of $B$.
We denote by $G'_W$ the generic Mumford-Tate group of $W$ with respect to the variation of Hodge structure $\mathbb{V}'$ and by ${G'}^{AH}_W$ the generic dR-absolute Mumford-Tate group of $W$ with respect to the absolute variation $({\mathbb{V}'}^{ \sigma})_{\sigma}$.

\begin{prop}\label{AHquotient}
For every closed irreducible subvariety $W \subset S$, we have $$G_{W'}^{AH} = {G'}_{W}^{AH}.$$
\end{prop}
\begin{proof}
By definition, the variation of Hodge structure $\mathbb{V}'$ is constant on the fibers of $f: S \to B$.
Since these fibers are connected it follows that also the absolute variation of Hodge structure is constant, which proves the proposition.
\end{proof}

\begin{thm}\label{reducepoints}
\begin{enumerate}[(i)]
\item \label{reducedeligne}
Suppose Deligne's conjecture \ref{conjdel1} holds for special points in $\bar{\mathbb{Q}}$-motivic variations. Then it holds for all subvarieties in $\bar{\mathbb{Q}}$-absolute variations of geometric origin.
\item \label{reduceabssp}
Suppose that Conjecture \ref{conjabssp} holds for special points in $\bar{\mathbb{Q}}$-motivic variations. Then it holds for all special subvarieties in $\bar{\mathbb{Q}}$-absolute variations of geometric origin.
\item \label{fodgalconj}
Suppose that special points in $\bar{\mathbb{Q}}$-motivic variations are defined over $\bar{\mathbb{Q}}$ and all their Galois conjugates are special. Then the same holds for all special subvarieties in $\bar{\mathbb{Q}}$-absolute variations of geometric origin.
\item \label{reducebialg}
Suppose that all $\bar{\mathbb{Q}}$-bialgebraic points in $\bar{\mathbb{Q}}$-motivic variations are dR-absolutely special. Then all $\bar{\mathbb{Q}}$-bialgebraic subvarieties in $\bar{\mathbb{Q}}$-absolute variations of geometric origin are dR-absolutely special.
\end{enumerate}
\end{thm}
\begin{proof}
Let $(\mathbb{V}^{\sigma})_{\sigma}$ be a $\bar{\mathbb{Q}}$-variation on $S$ of geometric origin.
For a subvariety $Z\subset S$, we construct a morphism $f: S \to B$ over $\bar{\mathbb{Q}}$ and a $\bar{\mathbb{Q}}$-motivic variation $({\mathbb{V}'}^{ \sigma}_B)_{\sigma}$ as described above.
By construction, the image of $Z$ under the period map $\Phi'$ is a point. As the map $\Psi$ is quasi-finite, we see that $x = f(Z)$ is a point of $B$.
We show that the various conjectures for $Z$ can be reduced to the ones for the point $x \in B(\mathbb{C})$.

\begin{enumerate}[(i)]
\item
Let $Z \subset S$ be a closed irreducible subvariety. If $Y$ is a special subvariety containing $Z$ with $G_Y = G_Z$, then Lemma \ref{lem2} shows that $G_Y^{AH} = G_Z^{AH}$. Hence we may assume that $Z$ is special.
By assumption, Deligne's conjecture holds for the special point $x$ in the $\bar{\mathbb{Q}}$-motivic variation $({\mathbb{V}'}^{ \sigma}_B)_{\sigma}$ on $B$. Thus we have the equality $G_x = G_x^{AH}$. Proposition \ref{AHquotient} now implies that $G'_Z = {G'}_Z^{AH}$. It follows that $$G_Z = H_Z\cdot G'_Z = H_Z \cdot {G'}_Z^{AH} = G_Z^{AH}.$$
\item \label{argument(ii)}
Let $Z \subset S$ be a special subvariety. We have to prove that if the special point $x \in B(\mathbb{C})$ is dR-absolutely special for $({\mathbb{V}'}^{ \sigma}_B)_{\sigma} $, then $Z$ is dR-absolutely special for $(\mathbb{V}^{\sigma})_{\sigma}$. Indeed, if $Z \subset W \subset S$ is such that $G_Z^{AH}=G_W^{AH}$ then setting $W' :=f(W)$ we obtain $G^{AH}_{W'} = G^{AH}_x$. We may assume that $W$ is special. As $x$ is dR-absolutely special by assumption, we get $W' =\{x\}$. Since $W$ was assumed to be special, it is thus an irreducible component of $f^{-1}(\{x\})$, as is $Z$. We conclude that $W=Z$ and $Z$ is dR-absolutely special.
\item
Let $Z\subset S$ be a special subvariety. By assumption, the special point $x=f(Z)$ is defined over $\bar{\mathbb{Q}}$ and its Galois conjugates $x^{\sigma} \in B^{\sigma} $ are special points.
Since $f$ is defined over $\bar{\mathbb{Q}}$, so is $Z$. Similarly, $Z^{\sigma}$ is an irreducible component of $(f^{\sigma})^{-1}(x^{\sigma})$ and since $x^{\sigma}$ is special it follows from Proposition \ref{kernel} that $Z^{\sigma} \subset S^{\sigma}$ is a special subvariety.
\item
It follows from the factorization of the period map for $\mathbb{V}'$ that if $Z \subset S$ is $\bar{\mathbb{Q}}$-bialgebraic, then $x \in B(\bar{\mathbb{Q}})$ is a $\bar{\mathbb{Q}}$-bialgebraic point. By assumption, $x$ is dR-absolutely special. Arguing as in (\ref{argument(ii)}), we see that $Z \subset S$ is dR-absolutely special.
\end{enumerate}

\end{proof}

\begin{rem}
The first part of Theorem \ref{reducepoints}(\ref{fodgalconj}) was already proven in \cite{KOU}. We emphasize that we follow the same strategy, except that their proof uses an argument where one is forced to change the $\bar{\mathbb{Q}}$-structure on $\mathcal{V}$ in order to prove that $\mathcal{V}'$ is defined over $\bar{\mathbb{Q}}$. By doing so, one also changes the notion of dR-absolute Hodge cycles.
As a consequence, their proof does not allow the conclusion for Galois conjugates in the second part of Theorem \ref{reducepoints}(\ref{fodgalconj}), let alone a proof of the other parts of Theorem \ref{reducepoints}.
The formalism of dR-absolutely special subvarieties allows us to show that $\mathcal{V}'$ is defined over $\bar{\mathbb{Q}}$ without affecting the $\bar{\mathbb{Q}}$-structure.
\end{rem}

\bibliographystyle{alpha}
\bibliography{Abssp_arXiv_new}

\begin{thebibliography}{CMSP17}

\bibitem[And92]{Andre}
Yves Andr\'e.
\newblock Mumford-tate groups of mixed {H}odge structures and the theorem of
  the fixed part.
\newblock {\em Compositio Mathematica}, 82(1):1--24, 1992.

\bibitem[BBT19]{BBT}
Benjamin Bakker, Yohan Brunebarbe, and Jacob Tsimerman.
\newblock o-minimal {G}{A}{G}{A} and a conjecture of {G}riffiths, 2019.
\newblock arXiv:1811.12230.

\bibitem[BKT20]{BKT}
Benjamin Bakker, Bruno Klingler, and Jacob Tsimerman.
\newblock Tame topology of arithmetic quotients and algebraicity of {H}odge
  loci.
\newblock {\em J. Amer. Math. Soc.}, 33(4):917--939, 2020.

\bibitem[BKU21]{BKU}
Gregorio Baldi, Bruno Klingler, and Emmanuel Ullmo.
\newblock On the distribution of the {H}odge locus, 2021.
\newblock arXiv:2107.08838.

\bibitem[CDK95]{CDK}
Eduardo Cattani, Pierre Deligne, and Aroldo Kaplan.
\newblock On the locus of {H}odge classes.
\newblock {\em J. Amer. Math. Soc.}, 8(2):483--506, 1995.

\bibitem[CMSP17]{periodbook}
James Carlson, Stefan Müller-Stach, and Chris Peters.
\newblock {\em Period Mappings and Period Domains}.
\newblock Cambridge Studies in Advanced Mathematics. Cambridge University
  Press, 2 edition, 2017.

\bibitem[Del70]{Deligneregular}
Pierre Deligne.
\newblock {\em \'{E}quations diff\'{e}rentielles \`a points singuliers
  r\'{e}guliers}.
\newblock Lecture Notes in Mathematics, Vol. 163. Springer-Verlag, Berlin-New
  York, 1970.

\bibitem[Del82]{Deligne}
Pierre Deligne.
\newblock Hodge cycles on abelian varieties.
\newblock In {\em Hodge cycles, motives, and {S}himura varieties}, volume 900
  of {\em Lecture Notes in Mathematics}, pages 9--100. Springer-Verlag,
  Berlin-New York, 1982.

\bibitem[DM82]{DM}
Pierre Deligne and James~S. Milne.
\newblock Tannakian categories.
\newblock In {\em Hodge cycles, motives, and {S}himura varieties}, volume 900
  of {\em Lecture Notes in Mathematics}, pages 101--228. Springer-Verlag,
  Berlin-New York, 1982.

\bibitem[GR03]{Grothendieck}
Alexander Grothendieck and Michel Raynaud.
\newblock {\em Rev\^{e}tements \'{e}tales et groupe fondamental ({SGA} 1)},
  volume~3 of {\em Documents Math\'{e}matiques (Paris)}.
\newblock Soci\'{e}t\'{e} Math\'{e}matique de France, Paris, 2003.
\newblock S\'{e}minaire de g\'{e}om\'{e}trie alg\'{e}brique du Bois Marie
  1960--61. [Lecture Notes in Math., 224, Springer, Berlin].

\bibitem[Kli17]{Klingleratypical}
Bruno Klingler.
\newblock {H}odge loci and atypical intersections: conjectures, 2017.
\newblock arXiv:1711.09387.

\bibitem[Kli21]{KlinglerICM}
Bruno Klingler.
\newblock Hodge theory, between algebraicity and transcendence, 2021.
\newblock arXiv:2112.13814.

\bibitem[KO21]{KO}
Bruno Klingler and Ania Otwinowska.
\newblock On the closure of the {H}odge locus of positive period dimension.
\newblock {\em Inventiones mathematicae}, 03 2021.

\bibitem[KOU20]{KOU}
Bruno Klingler, Ania Otwinowska, and David Urbanik.
\newblock On the fields of definition of {H}odge loci, 2020.
\newblock arXiv:2010.03359.

\bibitem[KUY18]{KUY}
Bruno Klingler, Emmanuel Ullmo, and Andrei Yafaev.
\newblock Bi-algebraic geometry and the {A}ndr\'{e}-{O}ort conjecture.
\newblock In {\em Algebraic geometry: {S}alt {L}ake {C}ity 2015}, volume~97 of
  {\em Proc. Sympos. Pure Math.}, pages 319--359. Amer. Math. Soc., Providence,
  RI, 2018.

\bibitem[Mil05]{MilneIntro}
J.~S. Milne.
\newblock Introduction to {S}himura varieties.
\newblock In {\em Harmonic analysis, the trace formula, and {S}himura
  varieties}, volume~4 of {\em Clay Math. Proc.}, pages 265--378. Amer. Math.
  Soc., Providence, RI, 2005.

\bibitem[Moo17]{Moonen}
Ben Moonen.
\newblock Families of motives and the {M}umford-{T}ate conjecture.
\newblock {\em Milan J. Math.}, 85(2):257--307, 2017.

\bibitem[Sch73]{Schmid}
Wilfried Schmid.
\newblock Variation of {H}odge structure: The singularities of the period
  mapping.
\newblock {\em Inventiones mathematicae}, 22:211--320, 1973.

\bibitem[SS16]{SchnellSaito}
Morihiko Saito and Christian Schnell.
\newblock Fields of definition of {H}odge loci.
\newblock In {\em Recent advances in {H}odge theory}, volume 427 of {\em London
  Math. Soc. Lecture Note Ser.}, pages 275--291. Cambridge Univ. Press,
  Cambridge, 2016.

\bibitem[UY11]{UY}
Emmanuel Ullmo and Andrei Yafaev.
\newblock A characterization of special subvarieties.
\newblock {\em Mathematika}, 57(2):263–273, 2011.

\bibitem[Voi06]{Voisin}
Claire Voisin.
\newblock {H}odge loci and absolute {H}odge classes.
\newblock {\em Compositio Mathematica}, 143, 06 2006.

\end{thebibliography}

\end{document}